\UseRawInputEncoding
\documentclass[11pt]{article}
\usepackage{ amssymb }
\usepackage{lipsum}

\usepackage{amsfonts}
\usepackage{amsthm}

\usepackage{epsfig}
\usepackage{graphicx}  
  
\usepackage{pdfsync}

\usepackage{caption}
\usepackage{subcaption}
 
\usepackage{amsmath,amsfonts,amssymb,amsthm,epsfig,epstopdf,titling,url,array}

\usepackage{hyperref}

\theoremstyle{plain}
\newtheorem{thm}{Theorem}[section]
\newtheorem{lem}[thm]{Lemma}
\newtheorem{prop}[thm]{Proposition}
\newtheorem{cor}{Corollary}

\theoremstyle{definition}

\theoremstyle{remark}

\theoremstyle{remark}

\usepackage{graphicx}
\usepackage{xcolor}
\definecolor{darkblue}{rgb}{0,0,0.5} 
\usepackage{transparent}

\DeclareMathOperator{\diam}{diam}

\begin{document}

\title{Asymptotic Dimension of Graphs of Groups and One Relator Groups.}

\author{Panagiotis Tselekidis}


\maketitle



\begin{abstract} 
We prove a new inequality for the asymptotic dimension of HNN-extensions. We deduce that the asymptotic dimension of every finitely generated one relator group is at most two, confirming a conjecture of A.Dranishnikov.\\
As further corollaries we calculate the exact asymptotic dimension of Right-angled Artin groups
and we give a new upper bound for  
the asymptotic dimension of fundamental groups of graphs of groups. 
\end{abstract}

\tableofcontents

\section{Introduction}
In 1993, M. Gromov introduced the notion of the asymptotic dimension of metric spaces (see \cite{Gr}) as an invariant of 
finitely generated groups. It can be shown that if two metric spaces are quasi isometric then they have the same asymptotic dimension.\\
The asymptotic dimension $asdimX$ of a metric space $X$ is defined as follows: $asdimX \leq n$ if and only if for every $R > 0$ there exists a uniformly bounded covering $\mathcal{U}$ of $X$ such that the R-multiplicity of $\mathcal{U}$ is smaller than or equal to $n+1$ (i.e. every R-ball in $X$ intersects at most $n+1$ elements of $\mathcal{U}$).\\
There are many equivalent ways to define the asymptotic dimension of a metric space. It turns out that the asymptotic dimension of an infinite tree is $1$ and the asymptotic dimension of $\mathbb{E}^{n}$ is $n$. \\


In 1998, the asymptotic dimension achieved particular prominence in geometric group theory after a paper of Guoliang Yu, (see \cite{Yu}) which proved the Novikov higher signature conjecture for manifolds whose fundamental group
has finite asymptotic dimension.\\
Unfortunately, not all finitely presented groups have finite asymptotic dimension. For example, Thompson's group $F$ has infinite
asymptotic dimension since it contains $\mathbb{Z}^{n}$ for all $n$.\\
However, we know for many classes of groups that they have finite asymptotic dimension, for instance, hyperbolic, relative hyperbolic, Mapping Class Groups of surfaces and one relator groups have finite asymptotic dimension (see \cite{BD08}, \cite{Os}, \cite{BBF}, \cite{Mats}). The exact computation of the asymptotic dimension of groups or finding the optimal upper bound is more delicate.\\
Another remarkable result is that of Buyalo and Lebedeva (see \cite{BL}) where in 2006 they established the following equality for hyperbolic groups:
\begin{center}
$asdim G = dim \partial_{\infty}G + 1$.
\end{center}

The inequalities of G.Bell and A.Dranishnikov (see \cite{BD04} and \cite{Dra08}) play a key role on finding an upper bound for the asymptotic dimension of groups. However, in some cases the upper bounds that the inequalities of G.Bell and A.Dranishnikov provide us are quite far from being optimal. An example is the asymptotic dimension of one relator groups.\\

In this paper we prove some new inequalities that can be a useful tool for the computation of the asymptotic dimension of groups. As an application we give the optimal upper bound for the asymptotic dimension of one relator groups which was conjectured by A.Dranishnikov. As a further corollary we calculate the exact asymptotic dimension of any Right-angled Artin group-this has been proven earlier by N.Wright \cite{Wr} by different methods.\\

The first inequality and one of the main results we prove is the following: 

\begin{thm}\label{1.1}
Let $G \ast_{N}$ be an HNN-extension of the finitely generated group $G$ over $N$. We have the following inequality

\begin{center}
$asdim\,G \ast_{N}  \leq max \lbrace asdim G, asdimN +1 \rbrace.$
\end{center}
\end{thm}


Next, we calculate the asymptotic dimension of the Right-angled Artin groups. To be more precise, let $\Gamma$ be a finite simplicial graph, we denote by $A(\Gamma)$ the \textit{Right-angled Artin group} (RAAG) associated to the graph $\Gamma$.
We set 
\begin{center}
$Sim(\Gamma)= max \lbrace n \mid $ $\Gamma$ contains the 1-skeleton of the standard $(n-1)$-simplex $\Delta^{n-1} \rbrace.$
\end{center} Then by applying Theorem \ref{1.1} we obtain the following:

\begin{thm}\label{1.2}
Let $\Gamma$ be a finite simplicial graph. Then,
$$asdimA(\Gamma)=Sim(\Gamma).$$
\end{thm}

%
%
%
%

In 2005, G.Bell, and A.Dranishnikov (see \cite{BD05}) gave a proof that the  asymptotic dimension of one relator groups is finite and also they gave an upper bound, namely the length of the relator plus one. Let $G= \langle S \mid r \rangle $ be a finitely generated one relator group such that $\mid\!r\!\mid = n$. Then
\begin{center}
$asdim\,G \leq n+1.$
\end{center}
To prove this upper bound G. Bell, and A. Dranishnikov used an inequality for the asymptotic dimension of HNN-extensions (see \cite{BD04}).\\
In particular, let $G$ be a finitely generated group and let $N$ be a subgroup of $G$. Then,
\begin{center}
$asdim\,G\ast_{N} \leq asdim\,G +1$.
\end{center}

In 2006, D. Matsnev  (see \cite{Mats}) proved a sharper upper bound for the asymptotic dimension of one relator groups. 
D. Matsnev proved the following: let $G= \langle S \mid r \rangle$ be a one relator group then 

\begin{center}
$asdim\,G \leq \lceil\!\frac{length(r)}{2}\!\rceil$.
\end{center}

Here by $\lceil\!a\!\rceil$ ($a \in \mathbb{R}$) we denote the minimal integer greater than
or equal to $a$.\\

Applying Theorem \ref{1.1} we answer a conjecture of A.Dranishnikov (see \cite{Dra}) giving the optimal upper bound for the asymptotic dimension of one relator groups.

\begin{thm}\label{1.3}
Let $G$ be a finitely generated one relator group. Then 
\begin{center}
$asdim\,G \leq 2$.
\end{center}
\end{thm}

We note that R. C. Lyndon (see \cite{Ly}) has shown that the \textit{cohomological dimension} of a torsion-free one-relator group is smaller than or equal to $2$. Our result can be seen as a large scale analog of this.\\
We note that the large scale geometry of one relator groups can be quite complicated, for example one relator groups can have very large isoperimetric functions (see e.g. \cite{Pl}).

It is worth noting that L.Sledd showed that the Assouad-Nagata dimension of any finitely generated $C^{\prime}(1/6)$ group is at most two (see \cite{Sl}).\\

Theorem \ref{1.3} combined with the results of M.Kapovich and B.Kleiner (see \cite{KK}) leads us to a description of the boundary of  hyperbolic one relator groups.

We determine also the one relator groups that have asymptotic dimension exactly two. We prove that every infinite finitely generated one relator group $G$ that is not a free group or a free product of a free group and a finite cyclic group has asymptotic dimension equal to 2 (Proposition \ref{3.5}).\\
We obtain the following:\\
\textbf{Corollary.} \textit{Let $G$ be finitely generated freely indecomposable one relator group which is not cyclic. 
Then
\begin{center}
$asdim\,G = 2$.
\end{center}}
Moreover, we describe the finitely generated one relator groups in the following corollary:\\
\textbf{Corollary.} \textit{Let $G$ be a finitely generated one relator group. Then one of the following is true}:\\
\textbf{(i)} \textit{$G$ is finite cyclic, and $asdim\,G = 0$} \\
\textbf{(ii)} \textit{$G$ is a nontrivial free group or a free product of a nontrivial free group and a finite cyclic group, and $asdim\,G = 1$}\\
\textbf{(iii)} \textit{$G$ is an infinite freely indecomposable not cyclic group or a free product of a nontrivial free group and an infinite freely indecomposable not cyclic group, and $asdim\,G = 2$.}

Using Theorem \ref{1.1} and an inequality of A.Dranishnikov about the asymptotic dimension of amalgamated products (see \cite{Dra08}) we obtain a more general theorem for the asymptotic dimension of fundamental groups of graphs of groups. 

\begin{thm}\label{1.4}
Let $(\mathbb{G}, Y)$ be a finite graph of groups with vertex groups $\lbrace G_{v} \mid v \in Y^{0} \rbrace$ and edge groups $\lbrace G_{e} \mid e \in Y^{1}_{+} \rbrace$. Then the following inequality holds:

\begin{center}
$asdim \pi_{1}(\mathbb{G},Y,\mathbb{T})  \leq max_{v \in Y^{0} ,e \in Y^{1}_{+}} \lbrace asdim G_{v}, asdim\,G_{e} +1 \rbrace.$
\end{center}

\end{thm}

    %
%
\textbf{Acknowledgments:} I would like to thank Panos Papasoglu for his valuable advices during the development of this research work.\\
I would also like to offer my special thanks to Mark Hagen and Richard Wade for their very useful comments.

\section{Asymptotic dimension of HNN-extensions.}


Let $X$ be a metric space and $\mathcal{U}$ a covering of $X$, we say that the covering $\mathcal{U}$ is \textit{$d$-bounded} or   \textit{$d$-uniformly bounded} if $sup_{U \in \mathcal{U}}\lbrace \diam  U \rbrace \leq d$. The \textit{Lebesgue number} $L(\mathcal{U})$ of the covering $\mathcal{U}$ is defined as follows:

\begin{center}
$L(\mathcal{U})= sup \lbrace \lambda \mid $ if $ A \subseteq X $ with $ \diam  A \leq \lambda $ then there exists $ U \in \mathcal{U} $ s.t. $ A\subseteq U \rbrace$.
\end{center} 
We recall that the order $ord (\mathcal{U})$ of the cover $\mathcal{U}$ is the smallest number $n$ (if it exists) such that each point of the space belongs to at most $n$ sets in the cover.\\
For a metric space $X$, we say that $(r,d)-dim X \leq n$ if for $r > 0$ there exists a $d$-bounded cover $\mathcal{U}$ of
$X$ with $ord(\mathcal{U}) \leq n + 1$ and with Lebesgue number $L(\mathcal{U}) > r$. We refer to such a
cover as an ($r,d$)-\textit{cover} of $X$.\\
The following proposition is due to G.Bell and A.Dranishnikov (see \cite{BD04}).

\begin{prop}\label{2.1}
For a metric space $X$, $asdimX \leq n $ if and only if there exists a function
$d(r)$ such that $(r,d(r))-dim X \leq n$ for all $r > 0$.
\end{prop}

We recall that the family $X_{i}$ of subsets of $X$ satisfies the inequality $asdim X_{i} \leq n$
\textit{uniformly} if for every $R>0$ there exists a $D$-bounded covering $ \mathcal{U}_{i}$ of $X_{i}$ with $R-mult(\mathcal{U}_{i}) \leq n+1$, for every $i$. For the proofs of the following theorems \ref{2.2} and \ref{2.3} see \cite{BD01}. 

\begin{thm}{(Infinite Union Theorem)}\label{2.2}
Let $X= \cup_{a} X_{a}$ be a metric space where the family $\lbrace X_{a} \rbrace$ satisfies the inequality $asdimX_{a} \leq n$ uniformly. Suppose further that for every $r>0$ there is a subset $Y_{r} \subseteq X$ with $asdim Y_{r} \leq n$ so that $d(X_{a} \setminus Y_{r} , X_{b} \setminus Y_{r}) \geq r$
whenever $X_{a} \neq X_{b}$. Then $asdimX \leq n $.
\end{thm}

\begin{thm}{(Finite Union Theorem)}\label{2.3}
For every metric space presented as a finite union $X= \cup_{i} X_{i}$ we have
\begin{center}
$asdimX = max\lbrace asdimX_{i} \rbrace$.
\end{center}

\end{thm}

A partition of a metric space $X$ is a presentation as a union $X= \cup_{i} W_{i}$ such that
$Int(W_{i})\cap  Int(W_{j}) = \varnothing $ whenever $i \neq j$. We denote by $\partial W_{i}$ the topological boundary of $W_{i}$ and by $Int(W_{i})$ the topological interior. We have that $ \partial W \cap Int(W) = \varnothing$. The boundary can be written as $$\partial W_{i}= \lbrace x \in X \mid d(x, W_{i})=d(x,X \setminus W_{i})=0    \rbrace.$$\\
For the proof of the following theorem see \cite{Dra08}.

\begin{thm}{(Partition Theorem)}\label{2.4}
Let X be a geodesic metric space. Suppose that for every $R > 0$ there is $d > 0$ and a partition $X= \cup_{i} W_{i}$ with $asdimW_{i}\leq n$ uniformly in $i$, and such that $(R,d)-dim(\cup_{i} \partial W_{i}) \leq n-1 $, where $\partial W_{i}$ is taken with
the metric restricted from $X$. Then $asdim X \leq n$.
\end{thm}

Let $G$ be a finitely generated group, $N$ a subgroup of $G$ and $\phi : N \rightarrow G$ a monomorphism. We set $\overline{G}=G\ast_{N}$ the HNN-extension of $G$ over the subgroup $N$ with respect to the monomorphism $\phi$. We fix a finite generating set $S$ for the group $G$. Then the set $\overline{S}=S\cup \lbrace t , t^{-1}\rbrace$ is a finite generating set for the group $\overline{G}$ and we set $C(\overline{G})=Cay(\overline{G},\overline{S})$ its Cayley graph.\\

\textit{Normal forms for HNN-extensions.}\\
We note that there exist two types of \emph{normal forms} for HNN-extensions, the right normal form and the left normal form. We are going to use both of them in this paper.\\
\textit{Right normal form:} Let $S_{N}$ and $S_{\phi(N)}$ be sets of representatives of \emph{right cosets} of $G/N$ and of $G/\phi(N)$ respectively. Then every $w \in \overline{G}$ has a unique normal form $w=gt^{\epsilon_{1}}s_{1}t^{\epsilon_{2}}s_{2}...t^{\epsilon_{k}}s_{k}$ where $g \in G$, $\epsilon_{i} \in \lbrace -1, 1 \rbrace$ and if $\epsilon_{i}=1 $ then $s_{i} \in S_{N}$, if $\epsilon_{i}=-1 $ then $s_{i} \in S_{\phi(N)}$.\\
\textit{Left normal form:} Let $_{N}S$ and $_{\phi(N)}S$ be sets of representatives of \emph{left cosets} of $G/N$ and of $G/\phi(N)$ respectively. Then every $w \in \overline{G}$ has a unique normal form $w=s_{1}t^{\epsilon_{1}}s_{2}t^{\epsilon_{2}}...s_{k}t^{\epsilon_{k}}g$ where $g \in G$, $\epsilon_{i} \in \lbrace -1, 1 \rbrace$ and if $\epsilon_{i}=1 $ then $s_{i} \in _{\phi(N)} S$, if $\epsilon_{i}=-1 $ then $s_{i} \in _{N}S$.\\
\textbf{Convention:} When we write a normal form we mean the right normal form, unless otherwise stated.\\

The group $\overline{G}= G\ast_{N}$ acts on its Bass-Serre tree $T$. There is
a natural projection $\pi: G\ast_{N} \rightarrow T$ defined by the action: $ \pi(g)=  gG$.\\
\begin{figure}
\begin{center} 
\includegraphics[scale=.5]{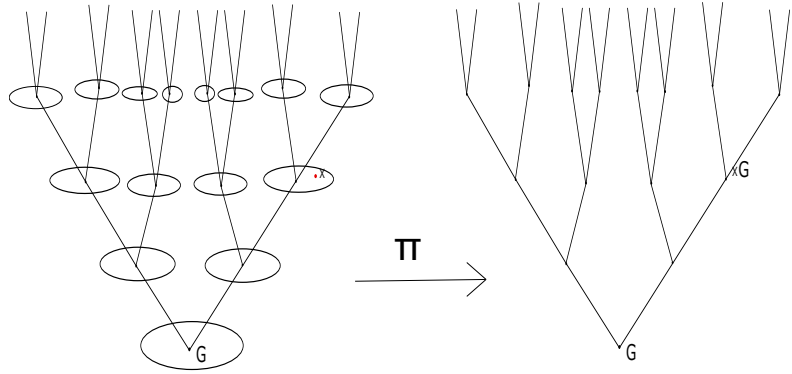}
\caption{An illustration of the projection $\pi: C(\overline{G},S) \rightarrow T$.}
\end{center}
\end{figure}

\begin{lem}\label{2.5}
The map $\pi: \overline{G} \rightarrow T$ extends to a simplicial map from the Cayley graph,
 $\pi: C(\overline{G},S) \rightarrow T$ which is 1-Lipschitz.
\end{lem}
 \begin{proof}
Let $g \in \overline{G}$ and $s \in \overline{S}$. Then the vertex $g$ is mapped to the vertex $\pi(g)=\pi(gs)=gG$.

If $s \in S$, then the edge $[g,gs]$ is mapped to the vertex $\pi(g)=\pi(gs)=gG$.

If  $s \in \lbrace t , t^{-1}\rbrace$, without loss of generality we
may assume that $s =t$, then the edge $[g,gs]$ is mapped to the edge $[\pi(g),\pi(gs)]=[gG,gtG]$ of $T$.\\

We observe that the  simplicial map $\pi: C(\overline{G}) \rightarrow T$  is 1-Lipschitz.
 \end{proof}

\begin{figure}
\begin{center} 
\includegraphics[scale=.2]{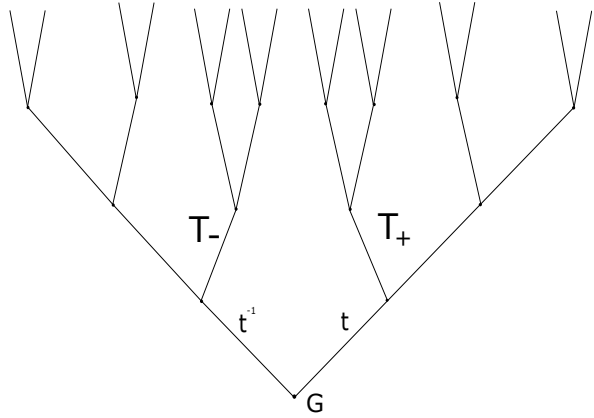}
\caption{An illustration of $T_+$ and $T_{-}$.}
\end{center}
\end{figure}

 The base vertex $G$ separates $T$ into two parts $T_{-} \setminus G$ and $T_{+} \setminus G$, where 
\begin{center}
$\pi^{-1}(T_{+})=\lbrace w \in \overline{G} \mid $ if $ w=gt^{\epsilon_{1}}s_{1}t^{\epsilon_{2}}s_{2}...t^{\epsilon_{k}}s_{k} $ is the normal form of $w$ then $ \epsilon_{1} = 1 \rbrace$
\end{center}
 and similarly
  
\begin{center}
$\pi^{-1}(T_{-})=\lbrace w \in \overline{G} \mid $ if $ w=gt^{\epsilon_{1}}s_{1}t^{\epsilon_{2}}s_{2}...t^{\epsilon_{k}}s_{k} $ is the normal form of $w$ then $ \epsilon_{1} = -1 \rbrace$.
\end{center}
We note that both $ T_{+} \setminus G$ and $T_{-} \setminus G $ are unions of connected components of $T$ and 
$ \pi^{-1}(T_{+})$ and $\pi^{-1}(T_{-}) $ are unions of connected components of $C(\overline{G})$. See figure 2 for an illustration of $T_-$ and $T_+$.\\

We consider the Bass-Serre tree $T$ as a metric space with the simplicial metric $\overline{d}$. If $Y$ is a graph, we denote by $Y^{0}$ or $V(Y)$ the vertices of $Y$.\\
For $u \in T^{0}$ we denote by $\mid\!u\!\mid$ the distance to the vertex with
label $G$.
We note that the distance of the vertex $wG$ from $G$ in the Bass-Serre tree $T$  equals to the length $l(w)$ of the normal form of $w$,  $\mid\!wG\!\mid=l(w)$. We denote by $l(w)$ the length of the normal form of $w$, we note that the length of both the right and the left normal form of $w$ is the same.\\

We recall that a \emph{full subgraph} of a graph $\Gamma$ is a subgraph formed from a subset of vertices $V$ and from all of the edges that have both endpoints in the subset $V$.\\
If $A$ is a subgraph of $\Gamma$ we define the \textit{edge closure} $E(A)$ of $A$ to be the full subgraph of $\Gamma$ formed from $V(A)$. Obviously, $V(E(A))=V(A)$.\\

We fix some notation on the Bass-Serre tree $T$ and on the Cayley graph.\\

\emph{In the tree $T$.} We denote by $B^{T}_{r}$
the $r-$ball in $T$ centered at $G$ ($r \in \mathbb{N}$).
There is a partial order on vertices of $T$ defined as follows: $v \leq u$ if and only if $v$
lies in the geodesic segment $[G,u]$ joining the base vertex $G$ with $u$. For $u \in T^{0}$ of
nonzero level (i.e. $u \neq G$) and $r > 0$ we set

\begin{center}
$T^{u}=  E(\lbrace v \in T^{0} \mid u \leq v \rbrace)$, $B^{u}_{r}=E(\lbrace v \in T^{u} \mid $ s.t. $ \mid\!v\!\mid \leq  \mid\!u\!\mid + r   \rbrace)$. 
\end{center}

For every vertex $u \in T^{0}$ represented by a coset $g_{u}G$ we have the
equality $B^{u}_{r}= g_{u}B^{T}_{r} \cap T^{u} $. We also observe that $B^{u}_{r}=E(\lbrace v \in T^{u} \mid $ s.t. $ \overline{d}(v,u) \leq  r   \rbrace)$. See figure 3 for an illustration of the sets $T^u$ and $B^u_r$\\


\begin{figure}
\centering
\begin{subfigure}{.5\textwidth}
  \centering
  \includegraphics[width=.99\linewidth]{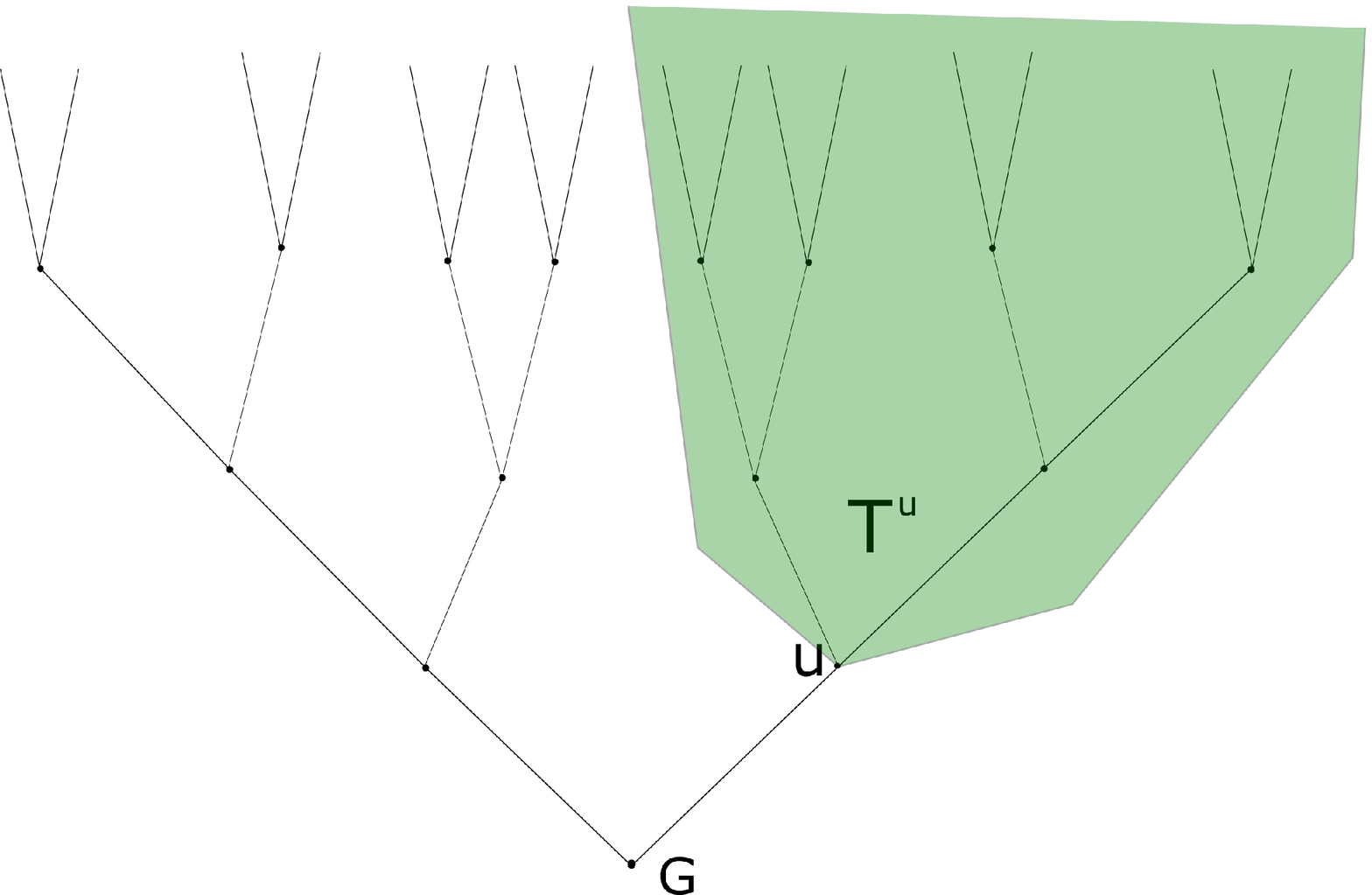}
  \caption{Here we have an illustration of $T^{u}$}
  \label{fig:sub1}
\end{subfigure}%
\begin{subfigure}{.5\textwidth}
  \centering
  \includegraphics[width=.99\linewidth]{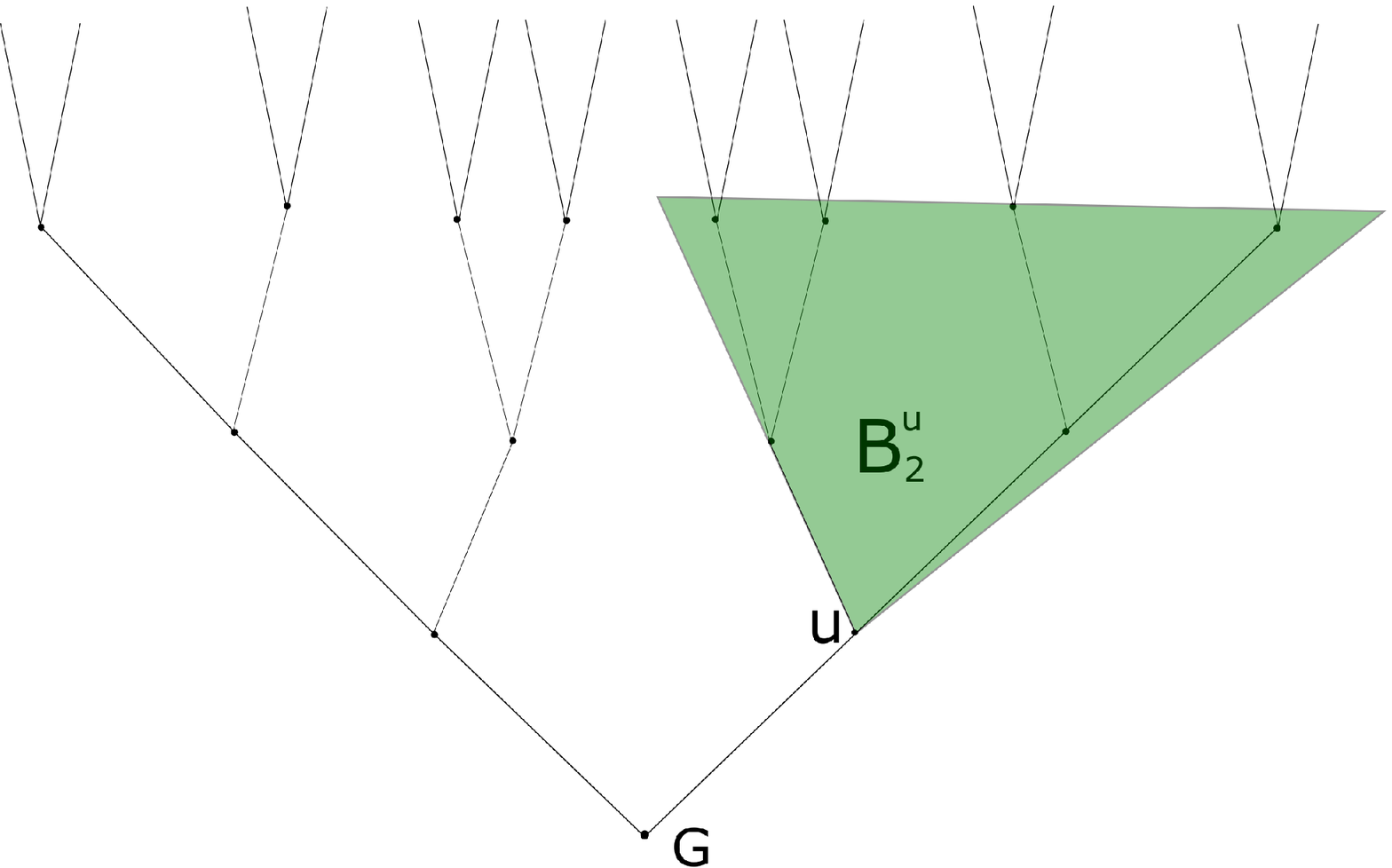}
  \caption{Here we have an illustration of $B^{u}_{r}$, where $r=2$. }
  \label{fig:sub2}
\end{subfigure}
\caption{}
\label{fig:test}
\end{figure}

\emph{In the Cayley graph.} 
For $R \in \mathbb{N}$, let $$M_{R}=  \lbrace  g \in \overline{G} \mid dist(g,N\cup\phi(N))=R \rbrace.$$ 


Let $u=g_{u}G$, we set $M_{R}^{u}=g_{u}M_{R} \cap \pi^{-1} (T^{u})$. We observe that $\pi(M^{u}_{R}) \subseteq B^{u}_{R} $ since $\pi$ is 1-Lipschitz.\\








\begin{figure}
\centering
\begin{subfigure}{.5\textwidth}
  \centering
  \includegraphics[width=.99\linewidth]{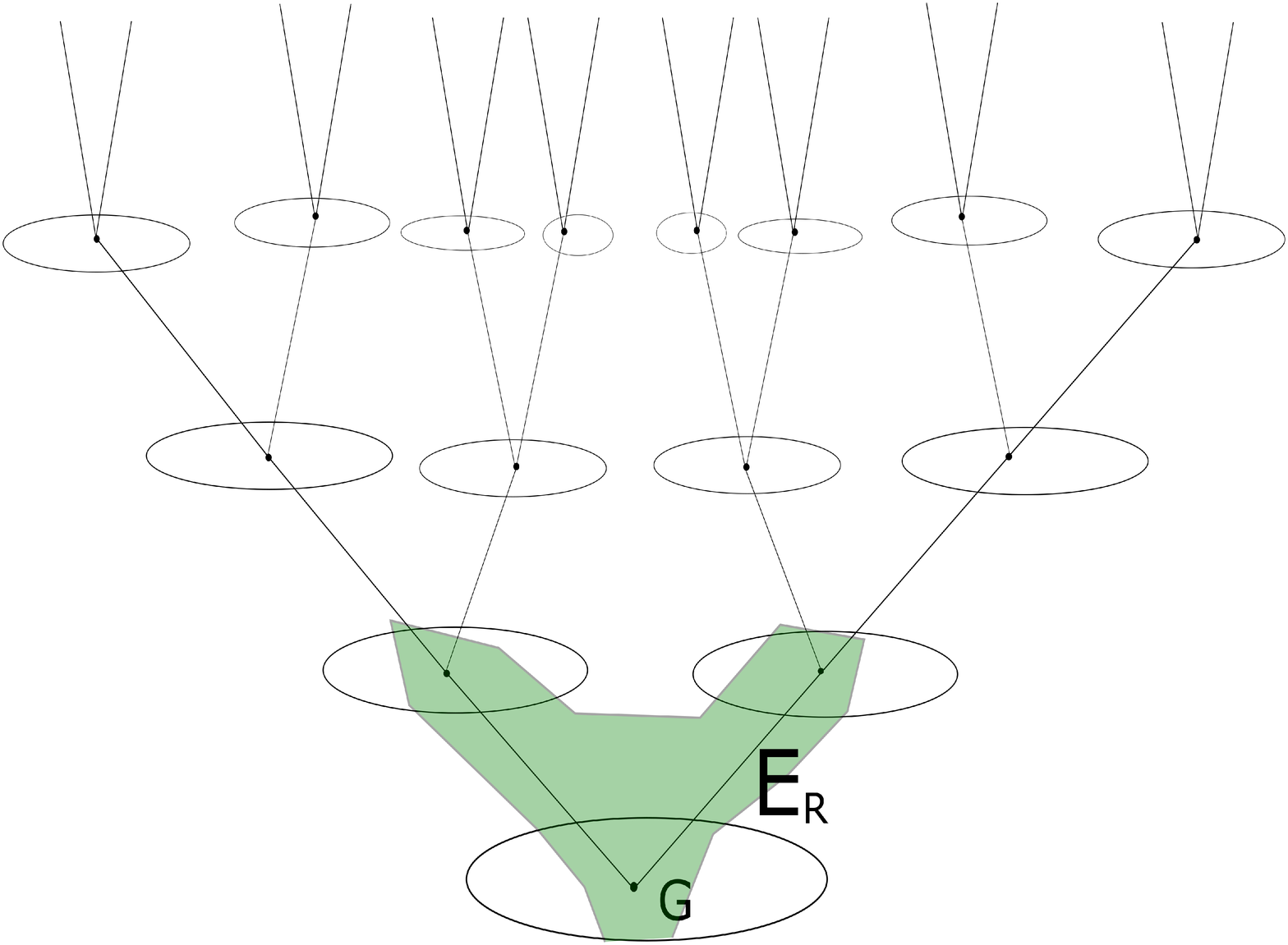}
  \caption{An illustration of $E_R$.}
  \label{fig:sub1}
\end{subfigure}%
\begin{subfigure}{.5\textwidth}
  \centering
  \includegraphics[width=.99\linewidth]{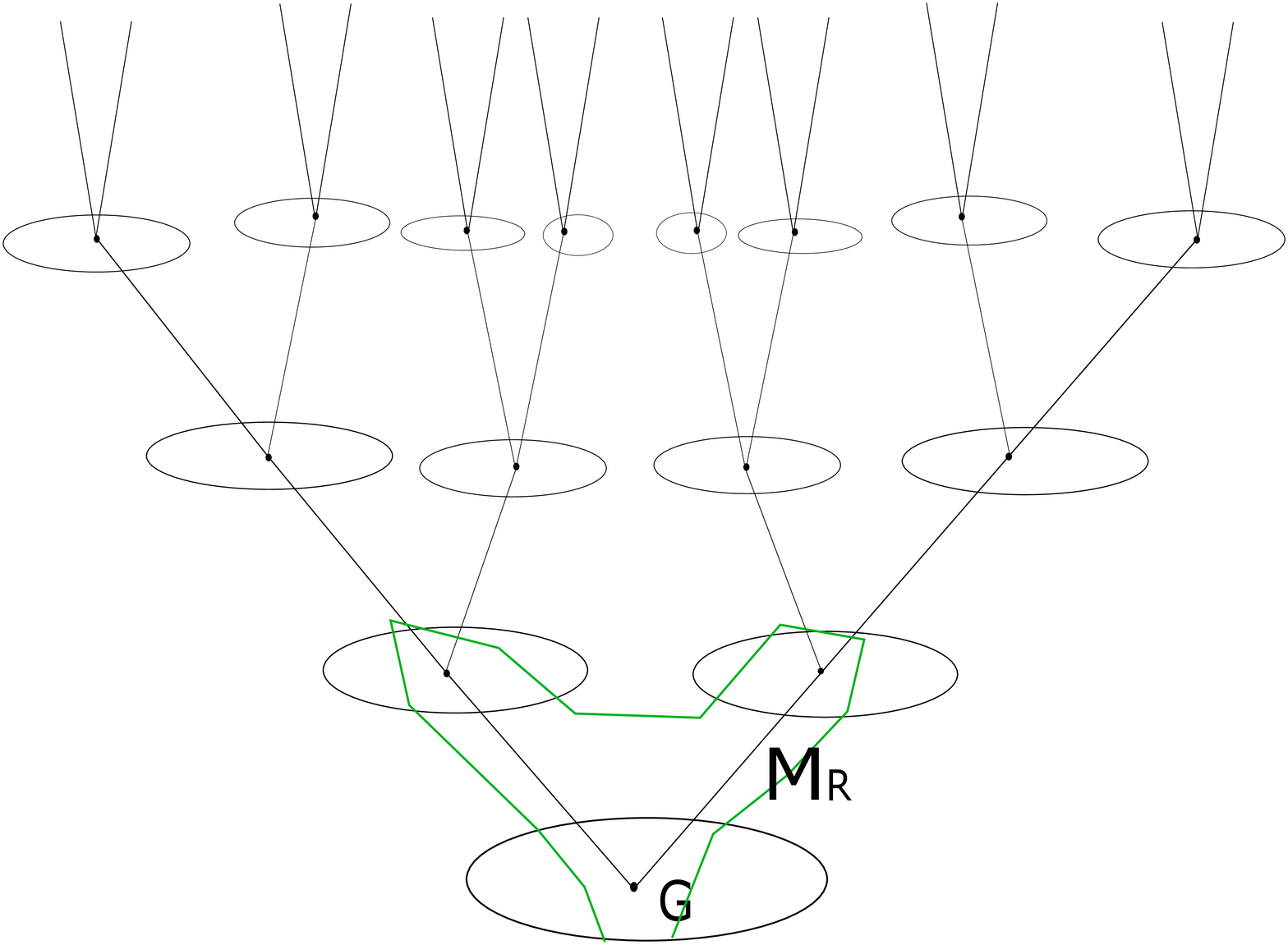}
  \caption{An illustration of $M_R$. }
  \label{fig:sub2}
\end{subfigure}
\caption{}
\label{fig:test}
\end{figure}

Let $u=g_{u}G$, we set $E_{R}=E(N_{R}(N \cup \phi(N)))$ and  $$E_{R}^{u}=g_{u}E_{R} \cap \pi^{-1} (T^{u}).$$ Obviously, $M_{R}^{u} \subseteq E_{R}^{u} \subseteq \pi^{-1}(B^{u}_{R}$).\\


\textbf{Convention:} We associate every $u \in T^{0}$ to an element $g_u \in \overline{G}$ such that the following two conditions hold:\\
(i) $u= g_u G$ .\\ 
(ii) if the left normal form of $g_u $ is $s_{1}t^{\epsilon_{1}}s_{2}t^{\epsilon_{2}}...s_{k}t^{\epsilon_{k}}g $ then $g=1_{\overline{G}}$.\\
We see that in this way we may define a bijective map from $T^{0}$ to the set $\mathcal{G}_T$ which consists of the elements of $\overline{G}$ such that conditions (i) and (ii) hold.
\begin{prop}\label{2.6}
If $4 < 4R \leq r$, and the distinct vertices $u, u^{\prime} \in T^0$, satisfy $ \mid\!u\!\mid ,  \mid\!u^{\prime}\!\mid  \in \lbrace nr \mid n \in \mathbb{N}\rbrace$ then
$$d(M_{R}^{u},M_{R}^{u^{\prime}}) \geq 2 R.$$ 
\end{prop}
\begin{proof}
We distinguish two cases. See figure 5(a) and figure 5(b) for the case 1 and case 2 respectively.

\begin{figure}
\centering
\begin{subfigure}{.5\textwidth}
  \centering
  \includegraphics[width=.99\linewidth]{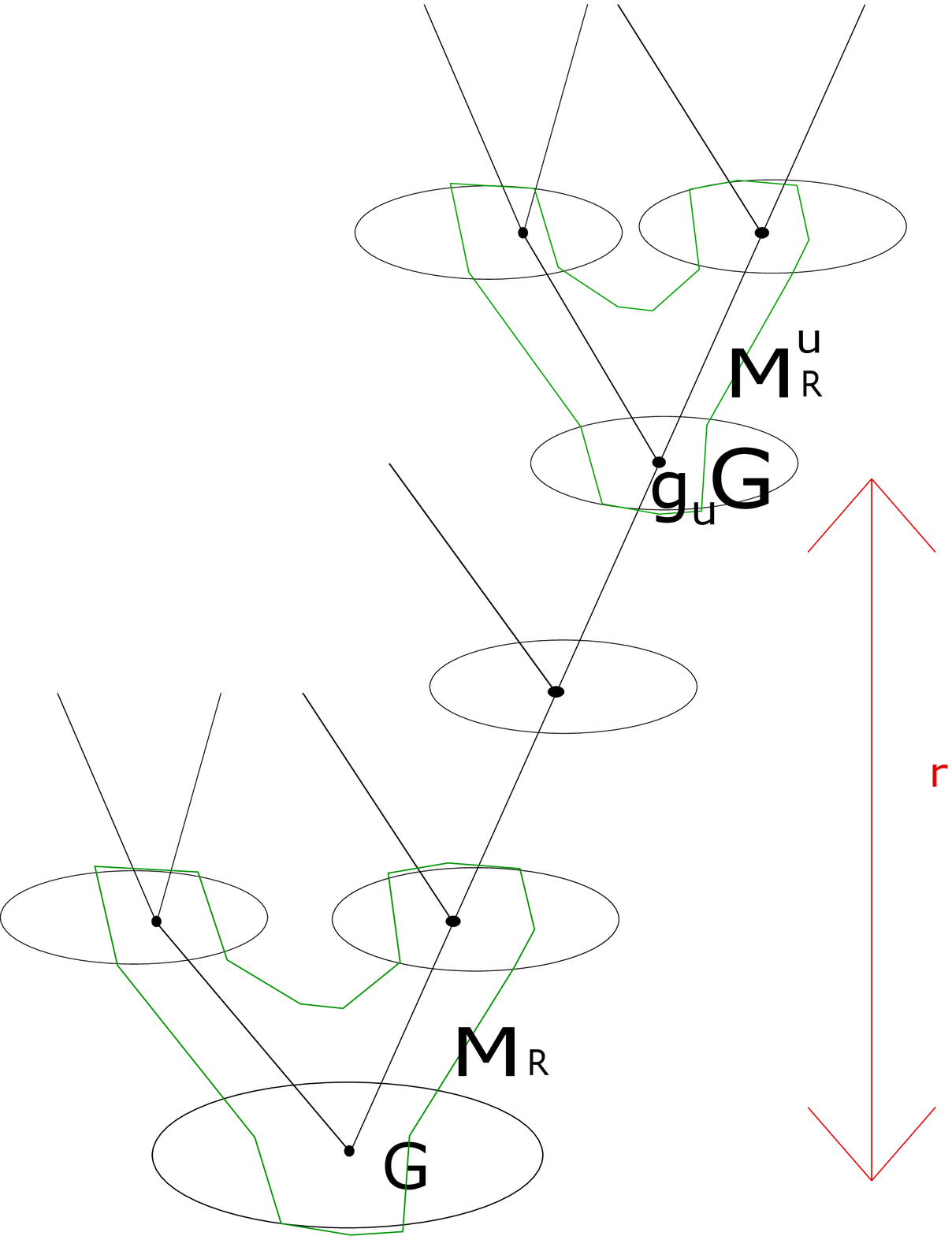}
  \caption{An illustration of case 1 of proposition \ref{2.6}.}
  \label{fig:sub1}
\end{subfigure}%
\begin{subfigure}{.5\textwidth}
  \centering
  \includegraphics[width=1.1\linewidth]{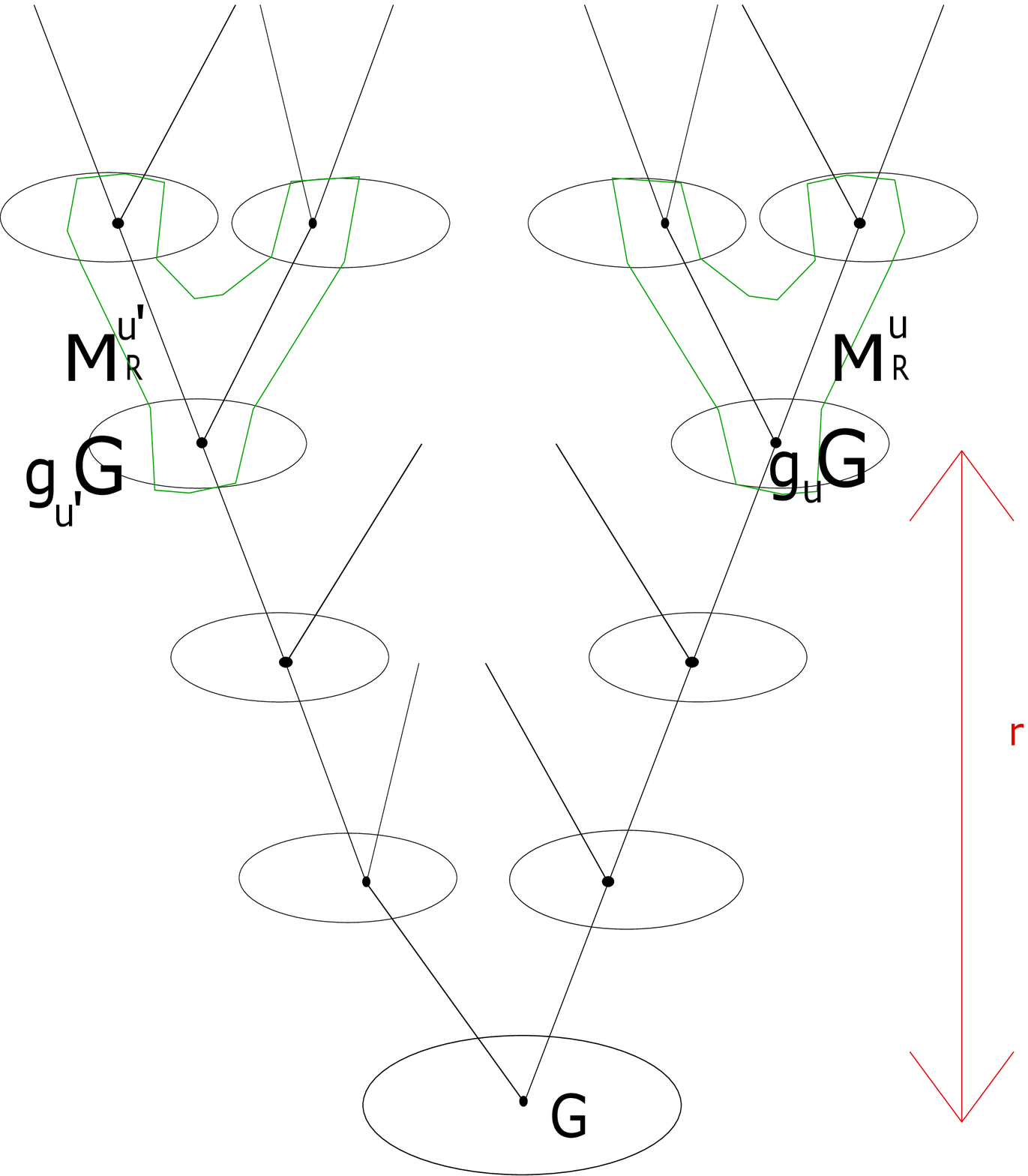}
  \caption{An illustration of case 2 of proposition \ref{2.6}. }
  \label{fig:sub2}
\end{subfigure}
\caption{In figure (a) we have $u^{\prime}=G$}
\label{fig:test}
\end{figure}

\textbf{Case 1:} $\mid\!u\!\mid \neq \mid\!u^{\prime}\!\mid$. We recall that every path $\gamma$ in $C(\overline{G})$ projects to a path $\pi(\gamma)$ in the tree $T$.
Then since 
\begin{center}
$M_{R}^{u}=g_{u}M_{R}\cap \pi^{-1}(T^{u})
 \subseteq \pi^{-1}(B_{R}^{u})$,
\end{center} 
 \begin{center}
  $M_{R}^{u^{\prime}}=g_{u^{\prime}}M_{R} \cap \pi^{-1}(T^{u^{\prime}})
 \subseteq \pi^{-1}(B_{R}^{u^{\prime}})$
 
  \end{center} and $\pi$ is 1-Lipschitz we have that  

\begin{center}
 $d(M^{u}_{R}
,M^{u^{\prime}}_{R}) \geq \overline{d}(B_{R}^{u},B_{R}^{u^{\prime}})
) \geq r - R \geq 3R.$
 \end{center}

\textbf{Case 2:} $\mid\!u\!\mid = \mid\!u^{\prime}\!\mid$ ($u \neq u^{\prime}$). We denote by $\zeta_{0}$ the last vertex of the common geodesic segment $[G,\zeta_{0}]$ of the geodesics $[G,u]$ and $[G,u^{\prime}]$. We observe that $\overline{d}(u,\zeta_{0}),\overline{d}(u^{\prime},\zeta_{0}) \geq 1$.\\
Let $x \in M_{R}^{u}$, $y \in M_{R}^{u^{\prime}}$ and let $\gamma$ be a geodesic from $x$ to $y$. Then the path $\pi(\gamma)$ passes through the vertices $u$, $u^{\prime}$ and $\zeta_{0}$. So the geodesic $\gamma$ intersects both $g_{u}(N \cup \phi(N))$ and $g_{u^{\prime}}(N \cup \phi(N))$.
Hence
\begin{center}
$d(x,y) \geq dist(x,g_{u}(N \cup \phi(N)) ) + dist(y,g_{u^{\prime}}(N \cup \phi(N)) )+ length([\zeta_{0},u^{\prime}]) + length([\zeta_{0},u]) \geq R+R+2=2(R+1).$
\end{center}
\end{proof}



For $w \in G\ast_{N}$, we denote by $\parallel\!w\!\parallel$ the distance from $w$ to $1_{\overline{G}}$ in the Cayley graph $Cay(\overline{G},\overline{S})$.

\begin{figure}
\centering
\begin{subfigure}{.5\textwidth}
  \centering
  \includegraphics[width=.9\linewidth]{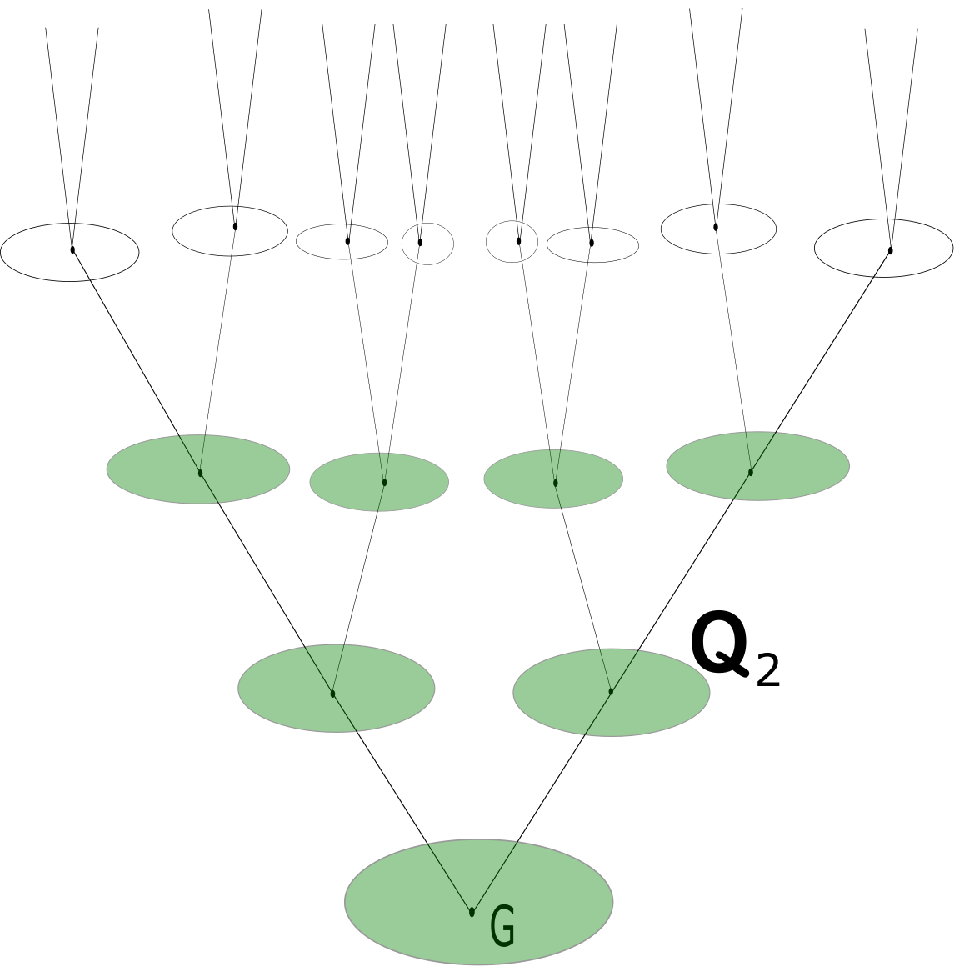}
  \caption{An illustration of $Q_m$, for $m=2$ (proposition \ref{2.8}).}
  \label{fig:sub1}
\end{subfigure}%
\begin{subfigure}{.5\textwidth}
  \centering
  \includegraphics[width=1\linewidth]{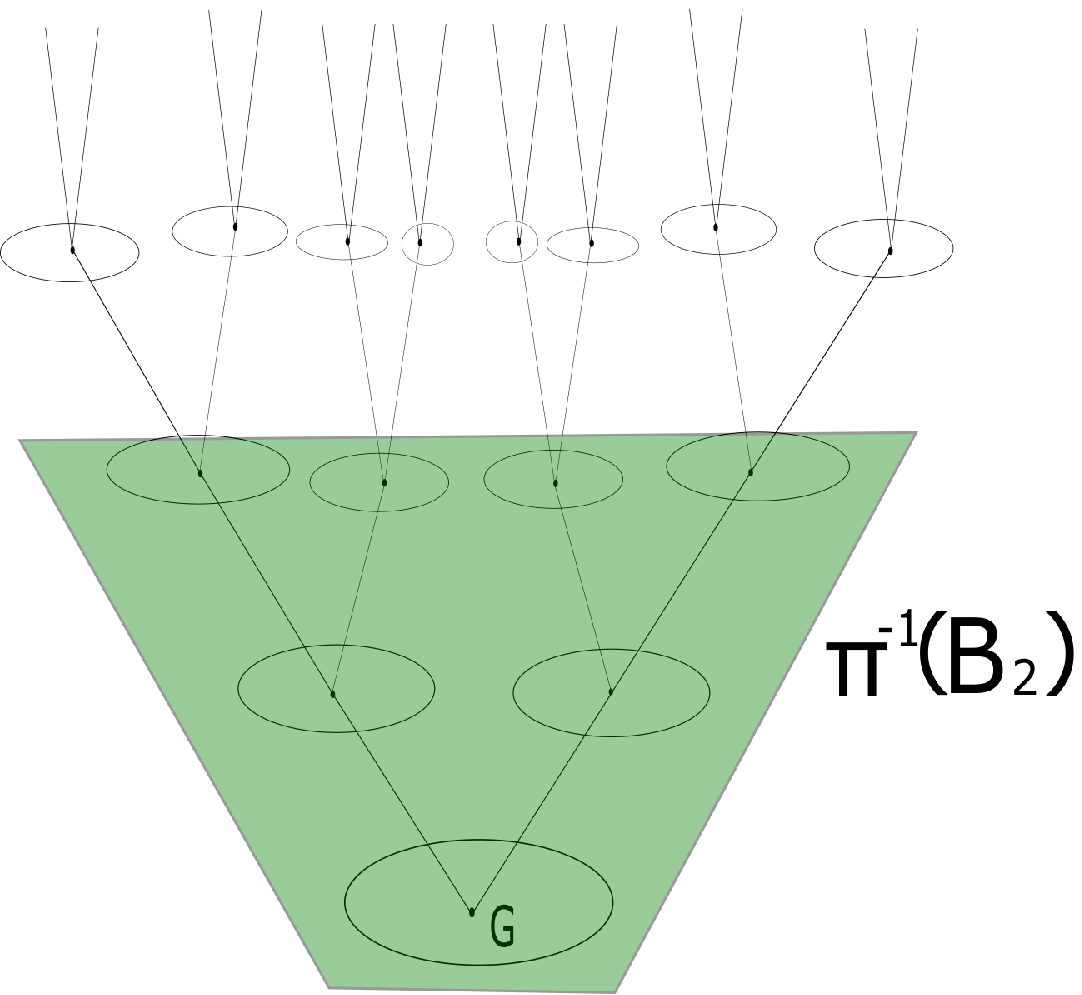}
  \caption{An illustration of $\pi^{-1}(B^{}_{r})$, where $r=2$.}
  \label{fig:sub2}
\end{subfigure}
\caption{We note that $Q_m= V(\pi^{-1}(B^{}_{r})).$}
\label{fig:test}
\end{figure}

\begin{lem}\label{2.7}
Let $w=gt^{\epsilon_{1}}s_{1}t^{\epsilon_{2}}s_{2}...t^{\epsilon_{k}}s_{k} $ be the normal form of $w$. Then

\begin{center}
$\parallel\!w\!\parallel \geq d(s_{k},N)$ if $\epsilon_{k}=1$ and $\parallel\!w\!\parallel \geq d(s_{k},\phi(N))$ if $\epsilon_{k}=-1$.
\end{center}
 
\end{lem}
\begin{proof}
Without loss of generality we assume that $\epsilon_{k}=1$. Let $$w=(\prod_{i_{0}=1}^{m_{0}}s_{i_{0}})t^{\epsilon_{1}}(\prod_{i_{1}=1}^{m_{1}}s_{i_{1}})t^{\epsilon_{2}}...t(\prod_{i_{k}=1}^{m_{k}}s_{i_{k}})$$ be a shortest presentation of $w$ in
the alphabet $\overline{S}$ (we note that $s_{i_{j}} \notin \lbrace t, t^{-1} \rbrace$). We set $\prod_{i_{j}=1}^{m_{j}}s_{i_{j}}=g_{j}$ for every $j \in \lbrace 1,...,k  \rbrace.$ Then $w=gt^{\epsilon_{1}}g_{1}t^{\epsilon_{2}}s_{2}...tg_{k} = w_{0}tg_{k}$.      \\
The first step when we rewrite $w$ in normal form starting from the previous presentation is to write $g_{k}=ns_{k}$ (where $n \in N$). Then $$\parallel\!w\!\parallel \geq \parallel\!g_{k}\!\parallel = \parallel\!ns_{k}\!\parallel =d(ns_{k},1) =d(s_{k},n^{-1}) \geq d(s_{k},N).$$
\end{proof}

We note that there exists an amalgamated product analogue of the following proposition proved by A.Dranishnikov in \cite{Dra08}.

\begin{prop}\label{2.8}
 Suppose that $asdim\,G \leq n$. Let 
 \begin{center}
 $Q_{m}= \lbrace w \in \overline{G} \mid  w=gt^{\epsilon_{1}}s_{1}t^{\epsilon_{2}}s_{2}...t^{\epsilon_{m}}s_{m} $ is the normal form of $w$ $ \rbrace.$
 \end{center} 
 
Then $asdim Q_{m} \leq n $, for every $m \in \mathbb{N}$.
\end{prop}

\begin{proof}
We set $P_{\lambda}= \lbrace w \in \overline{G} \mid l(w)= \lambda \rbrace$. To prove the statement of the proposition it is enough to show that $asdimP_{\lambda} \leq n$, for every $\lambda \in \mathbb{N}$. Indeed, since 

\begin{center}
 $Q_{m}= \cup _{i=1}^{m}P_{i}$ 
\end{center}

by the Finite Union Theorem we obtain that $asdimQ_{m} \leq n$.\\

\textbf{Claim:} For $\lambda \in \mathbb{N}$ we have $asdimP_{\lambda} \leq n.$\\
\textit{Proof of claim}: We use induction on $\lambda$. We have $P_{0}=G$, so $asdimP_{0} \leq n$. We observe that $P_{\lambda} \subseteq P_{\lambda-1}tG   \cup  P_{\lambda-1}t^{-1}G$. Using the Finite Union Theorem it suffices to show that $asdim(P_{\lambda} \cap   P_{\lambda-1}tG)  \leq n$ and   $asdim(P_{\lambda} \cap P_{\lambda-1}t^{-1}G)  \leq n$, we show the first.\\

To show that $asdimP_{\lambda} \cap P_{\lambda-1}tG  \leq n$ we use the Infinite Union Theorem. For $r>0$ we set $Y_{r}=P_{\lambda-1}tN_{r}(N)$. We claim that $$Y_{r} \subseteq N_{r+1}(P_{\lambda-1}).$$ Indeed, if $z \in Y_{r}$  then $z=z_{0}tz_{1}$, where $z_{0} \in P_{\lambda-1}$ and $z_{1} \in N_{r}(N)$. Since $z_{1} \in N_{r}(N)$ there exists $n \in N$ with $ d(n,z_{1}) \leq $ r.\\
So $z=z_{0}tnn^{-1}z_{1}= z_{0}\phi(n)tn^{-1}z_{1}$, and $$d(z,P_{\lambda-1}) \leq d(z, z_{0}\phi(n)) = \parallel\!tn^{-1}z_{1}\!\parallel \leq \parallel\!t\!\parallel + \parallel\!t^{-1}z_{1}\!\parallel \leq 1 + r.$$

Hence $Y_{r} \sim_{q.i.} P_{\lambda-1}$, so $asdimY_{r}\leq n$.\\

We consider the family $xtG$ where $x \in  P_{\lambda-1}$. For $xtG \neq ytG$, we have $d(xtG \setminus Y_{r}, ytG \setminus Y_{r}) = d(xtg,yth) =   \parallel\!g^{-1}t^{-1}x^{-1}yth\!\parallel $, where $g, h \in G \setminus N_{r}(N)$. The first step when we rewrite $g^{-1}t^{-1}x^{-1}yth$ in normal form is to replace $h=ns_{k}$, where $n \in N$ and $s_{k} \in S_{N}$, so $ g^{-1}t^{-1}x^{-1}yth = g^{-1}t^{-1}x^{-1}y\phi(n)ts_{k}$.\\
Since $h \in G \setminus N_{r}(N)$ we have that $\Vert\!s_{k}\!\Vert =\Vert\!n^{-1}h\!\Vert \geq d(h,N) \geq r$.

By lemma \ref{2.7} we obtain that $\parallel\!g^{-1}t^{-1}x^{-1}y\phi(n)ts_{k}\!\parallel \geq \parallel\!s_{k}\!\parallel \geq r$.\\ 

Finally, by observing that $xtG$ and $G $ are isometric we deduce that $asdim(xtG) \leq n$ uniformly. Since all the conditions of the Infinite Union Theorem hold we have that  $$asdim(P_{\lambda} \cap   P_{\lambda-1}tG)  \leq n$$ for every $\lambda \in \mathbb{N}.$

\end{proof}

We observe that $E(Q_{m}) = \pi^{-1}(B^{T}_{m})$  and $ Q_{m} = \overline{G} \cap \pi^{-1}(B^{T}_{m})$.\\


For $w \in \overline{G}$, we set $T^{w} = T^{\pi(w)} $, where $\pi(w)=wG$.

\begin{thm}\label{2.9}
Let $G \ast_{N}$ be an HNN-extension of the finitely generated group $G$ over $N$. We have the following inequality

\begin{center}
$asdim\,G \ast_{N}  \leq max \lbrace asdim G, asdimN +1 \rbrace.$
\end{center}

\end{thm}

\begin{proof} Let $n= max \lbrace asdim G, asdimN +1 \rbrace$. We denote by $\pi:C(\overline{G},S)  \rightarrow T$ the map of Lemma \ref{2.5}.\\ 
We recall that we denote by $l(g)$ the length of the normal form of $g$.\\


We will use the Partition Theorem (Thm \ref{2.4}). Let $R,r \in \mathbb{N}$ be such that $R>1$ and $r > 4R$.

We set,   
$$U_{r}= E[(\pi^{-1}(B^{T}_{r-1}) \cap E(\lbrace g \in \overline{G} \mid d(g,N \cup \phi(N)) \geq R   \rbrace))\cup(\bigcup_{u \in \partial B^{T}_{r} } E_{R}^{u})],$$
 where $E_{R}^{u}= g_{u}E(N_{R}(N \cup \phi(N))) \cap \pi^{-1}(T^{u})$.\\
We recall that $M_{R}=\lbrace g \in \overline{G} \mid d(g,N \cup \phi(N)) = R   \rbrace$.\\
Let $A_{R}$ be the collection of the edges between the elements of $M_{R} \subseteq U_{r}$.
We have that $A_{R} \subseteq U_{r}$. We define $V_r$ to be the set obtained by removing the interior of the edges of $A_{R}$ from $U_{r}$.\\ 
Formally we have that  
\begin{center}
$V_{r} =  U_{r} \setminus \lbrace interior(e) \mid e \in A_{R} \rbrace$.
\end{center}

\begin{figure}
\begin{center} 
\includegraphics[scale=.6]{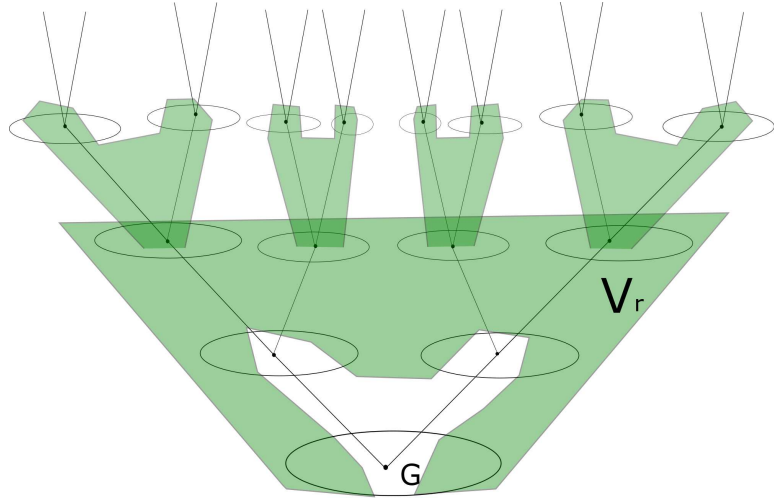}
\caption{An illustration of $V_r$.}
\end{center}
\end{figure}

See figure 7 for an illustration of the set $V_r$. 
We observe that the sets $U_{r}$ and $V_{r}$ are subgraphs of $C(\overline{G})$, $\partial U_{r}= \partial V_{r}$ and $ V_{r} \cap \overline{G} = U_{r} \cap \overline{G}$. Obviously, $\bigcup_{u \in \partial B^{T}_{r} } E_{R}^{u} \subseteq V_{r}.$
We also have
$$V_{r} \cap \overline{G} = (\overline{G} \cap \pi^{-1}(B^{T}_{r-1}) \cap E(\lbrace g \in \overline{G} \mid d(g,N \cup \phi(N)) \geq R   \rbrace))\cup( \overline{G} \cap \bigcup_{u \in \partial B^{T}_{r} }  E_{R}^{u}).$$ 
To be more precise,

  
\begin{center}
$V_{r} \cap \overline{G} = \lbrace  wx \in \overline{G}  \mid d(w,N \cup \phi(N)) \geq R $ and  if $ w=g_{0}t^{\epsilon_{1}}g_{1}...t^{\epsilon_{k}}g_{k} $ is the normal form of $w$ then  $ k \leq r -1$ or $g_{k}=1$ and $k=r$ , if $x \neq 1$ then $ k = r $, $g_{k}=1$  , $d(x, N \cup \phi(N)) \leq R         \rbrace $.
\end{center} 

For every vertex $u \in T^0$ satisfying $\mid\!u\!\mid  \in \lbrace nr \mid n \in \mathbb{N}\rbrace$, we define
\begin{center}
 $V^{u}_{r} = g_{u}V_{r} \cap \pi^{-1}(T^{u})$.
\end{center}

Obviously, the sets $V^{u}_{r}$ are subgraphs of $C(\overline{G})$ and $V^{u}_{r} \nsubseteq \overline{G}$. We observe that $V_{r}  \subseteq \pi^{-1}(B^{T}_{r+R})$, so $V_{r}^{u}  \subseteq \pi^{-1}(B^{u}_{r+R})$. Obviously, for every $h$ such that $h=g_{1}t^{\epsilon_{1}}g_{2}...t^{\epsilon_{r}}$ is the \emph{left normal form} of $h$ we have that:

 \begin{center}
  $(g_{u} M_{R} \cap \pi^{-1}(T^{g_{u}G}))\cup (g_{u} h M_{R} \cap \pi^{-1}(T^{g_{u}hG}))\subseteq \partial V_{r}^{u}$, where $l(h)=r$. ($\star$)
  \end{center} 



This can also be written as:
\begin{center}
  $M_{R}^{g_{u}G} \cup M_{R}^{g_{u}hG}\subseteq \partial V_{r}^{u}$. 
    \end{center}
We set $V_{r}^{G}=V_{r}$.

We consider the partition
\begin{equation}
C(\overline{G}) =\pi^{-1}(T) = (\bigcup_{\mid u \mid \in \lbrace nr \mid n \in \mathbb{N}_{+} \cup \lbrace 0 \rbrace \rbrace} V^{u}_{r}) \cup E(N_{R}(N \cup \phi(N)))
\end{equation}
We set,
\begin{center}
$Z = (\bigcup_{\mid u \mid \in \lbrace nr \mid n \in \mathbb{N}_{+} \rbrace \cup \lbrace  0  \rbrace} \partial V^{u}_{r}) \cup \partial E(N_{R}(N \cup \phi(N)))$.
\end{center}

We observe that if $V^{u}_{r} \cap V^{v}_{r} \neq \varnothing $, then either $u \leq v$ and $\mid\!u\!\mid +r = \mid\!v\!\mid$ or $u \geq v$
and $\mid\!v\!\mid +r = \mid\!u\!\mid$. 
If $V^{u}_{r} \cap V^{v}_{r} \neq \varnothing $ such that $u\leq v$ and $\mid\!u\!\mid +r = \mid\!v\!\mid$, then

\begin{center}
$V_{r}^{u} \cap V_{r}^{v} = M^{v}_{R} = \partial V_{r}^{u} \cap \partial V_{r}^{v}.$
\end{center}






We deduce that
\begin{center}
$Z = (\bigcup_{\mid u \mid \in \lbrace nr \mid n \in \mathbb{N}_{+} \rbrace  } M^{u}_{R})) \cup M_{R}$.
\end{center}


We will show that there exists $d>0$ such that $(R,d)-dimZ \leq n-1 $.\\ Since $M_{R}$ is quasi isometric to $N_{R}(N \cup \phi(N))$, which is quasi isometric to $N \cup \phi(N)$, we have that $asdim M_{R} \leq n-1$. Then for $R>0$ there exists a $(R,d)$-covering  $\mathcal{U}$ of $M_{R}$ with $ord(\mathcal{U})\leq n$.\\
In view of the proposition \ref{2.6} we have that the following covering

$$\mathcal{V}=\mathcal{U} \cup \bigcup_{\mid u \mid \in \lbrace nr \mid n \in \mathbb{N}_{+} \rbrace} (g_{u}\mathcal{U} \cap M_{R}^{u})$$

 is a $(R,d)$-covering of $Z$ with $ord(\mathcal{V})\leq n$. We conclude that $(R,d)-dimZ \leq n-1 $.\\

 
Next, we will show that $asdimV^{u}_{r} \leq n$ and $asdimN_{R}(N \cup \phi(N)) \leq n$ uniformly. This will complete our proof that all the conditions of the Partition Theorem are satisfied.\\
It suffices to show that $asdimV^{u}_{r} \leq n$ uniformly and $$asdimN_{R}(N \cup \phi(N))\leq n.$$
We observe that $V_{r} \subseteq \pi^{-1}(B^{T}_{r+R}) \subseteq N_{1}(Q_{r+R})$, so by the proposition \ref{2.8} we have that $asdimV^{u}_{r} \leq n$. Since the sets $V^{u}_{r}$ of our partition are isometric to each other we conclude that $asdimV^{u}_{r} \leq n$ uniformly.\\
Finally, $asdimN_{R}(N \cup \phi(N)) \leq n-1$ since  $N_{R}(N \cup \phi(N)))$  is quasi isometric to $N \cup \phi(N)$.\\

By the Partition Theorem (Thm \ref{2.4}), $asdimC(\overline{G}) = asdim \pi^{-1}(T) \leq n$.

\end{proof}
 

\subsection{Right-angled Artin groups.}
We use the following theorem of G.Bell, A.Dranishnikov, and J.Keesling (see \cite{BDK}).

\begin{thm}\label{2.10}
If $A$ and $B$ are finitely generated groups then $$asdim\, A \ast B = max \lbrace  asdimA , asdimB  \rbrace.$$
\end{thm}
 
Let $\Gamma$ be a finite simplicial graph with $n$ vertices, the \textit{Right-angled Artin group} (RAAG) $A(\Gamma)$ associated to the graph $\Gamma$ has the following presentation: 
\begin{center}
$ A(\Gamma)= \langle s_{u} $ , $ (u \in V(\Gamma))  \mid [s_{u},s_{v}] $ , $ ([u,v] \in E(\Gamma))    \rangle .$
\end{center}

By $[s_{u},s_{v}]=s_{u}s_{v}s_{u}^{-1}s_{v}^{-1}$ we mean the commutator.\\
We set $Val(\Gamma)= max \lbrace  valency(u) \mid u \in V(\Gamma) \rbrace$. By $valency(u)$ of a vertex $u$ we denote the number of edges incident to the vertex $u$.\\
Clearly, $Val(\Gamma) \leq rank(A(\Gamma))-1.$\\

If $\Gamma$ is a simplicial graph, we denote by $1-skel(\Gamma)$ the 1-skeleton of $\Gamma$. Recall that a \emph{full subgraph} of a graph $\Gamma$ is a subgraph formed from a subset of vertices $V$ and from all of the edges that have both endpoints in the subset $V$.

\textbf{Conventions:} Let $\Gamma$ be simplicial graph, $u \in V(\Gamma)$ and $e \in E(\Gamma)$. We denote by:\\
(i) $\Gamma \setminus \lbrace u \rbrace$ the full subgraph of $\Gamma$ formed from $V(\Gamma)\setminus \lbrace u \rbrace$.\\
(ii) $\Gamma \setminus e$ the subgraph of $\Gamma$ such that $V(\Gamma \setminus e)=V(\Gamma)$ and $E(\Gamma \setminus e)=E(\Gamma) \setminus \lbrace e \rbrace$.

\begin{lem}\label{2.11}
Let $\Gamma$ be a finite simplicial graph. Then 

\begin{center}
$asdimA(\Gamma)  \leq Val(\Gamma)+1.$
\end{center}
\end{lem}
\begin{proof}
Since theorem \ref{2.10} holds we observe that it suffices to show the statement of the lemma \ref{2.11} for connected simplicial graphs. We assume that $\Gamma$ is a connected simplicial graph.\\
We use induction on the $rank(A(\Gamma))$. For $rank(A(\Gamma))=1$ we have that $A(\Gamma)$ is 
the integers $\mathbb{Z}$, so the statement holds. We assume that the statement holds for every $k \leq n$ and we show that it holds for $n+1$ ($n+1 \geq 2$).\\

Let $\Gamma$ be a simplicial graph with $n+1$ vertices. We remove a vertex $u$ from the graph $\Gamma$ such that $valency(u)= Val(\Gamma) = m \geq 1$. Let's denote by $v_{i}$ ($i \in \lbrace 1,..., m \rbrace$) the vertices of $\Gamma$ which are adjacent to $u$.\\ 
We set $\Gamma^{\prime}=\Gamma \setminus \lbrace u \rbrace $. Obviously, $Val(\Gamma^{\prime}) \leq Val(\Gamma)$.\\
We denote by $Y$ the full subgraph of $\Gamma$ formed from $\lbrace  v_{1}, \ldots ,v_{m} \rbrace$.\\

We observe that the RAAG $A(\Gamma)$ is an HNN-extension of the RAAG $A(\Gamma^{\prime})$. To be more precise, we have that
\begin{center}
$A(\Gamma)=  A(\Gamma^{\prime})\ast_{A(Y)}.$ 
\end{center}


By Theorem \ref{2.9} we obtain that $$asdimA(\Gamma) \leq max\lbrace asdimA(\Gamma^{\prime})  , asdimA(Y)+1  \rbrace .$$ 

We observe that $  Val(Y)  \leq Val(\Gamma) -1$, so by the induction ($ rank(A(Y)) \leq n$) we obtain  $$asdimA(Y) \leq Val(Y) +1  \leq Val(\Gamma).$$

Since $rankA(\Gamma^{\prime}) = n$, by the induction we deduce that  $$asdimA(\Gamma^{\prime})  \leq Val(\Gamma^{\prime}) + 1 \leq Val(\Gamma) + 1.$$ Combining the three previous inequalities we obtain:

\begin{center}
$   asdimA(\Gamma) \leq max\lbrace  Val(\Gamma) + 1  , Val(\Gamma) + 1 \rbrace =   Val(\Gamma) + 1  . $
\end{center}




\end{proof}

Using the previous lemma we can compute the exact asymptotic dimension of $A(\Gamma)$. We note that this has already been computed by N.Wright \cite{Wr} using different methods.\\

We set 
\begin{center}
$Sim(\Gamma)= max \lbrace n \mid $ $\Gamma$ contains the 1-skeleton of the standard $(n-1)$-simplex $\Delta^{n-1} \rbrace.$
\end{center} 
Obviously if $\Gamma^{\prime} \subseteq \Gamma$, then $Sim(\Gamma^{\prime}) \leq Sim(\Gamma)$.


\begin{thm}\label{2.12}
Let $\Gamma$ be a finite simplicial graph. Then,

$$asdimA(\Gamma)=Sim(\Gamma).$$
\end{thm}
\begin{proof}
Since theorem \ref{2.10} holds we observe that it suffices to show the statement of Theorem \ref{2.12} for connected simplicial graphs. We assume that $\Gamma$ is a connected simplicial graph.\\
\textbf{Claim 1.} $Sim(\Gamma) \leq asdimA(\Gamma)$.\\
\textit{Proof of claim 1:} Let $Sim(\Gamma)=n$. We observe that $\mathbb{Z}^{n} = A(S_{n-1}) \leq A(\Gamma)$.
It follows that
\begin{center}
$n= asdim\mathbb{Z}^{n} \leq asdimA(\Gamma)$.  
\end{center} 
\textbf{Claim 2.} $asdimA(\Gamma ) \leq Sim(\Gamma)$.\\
\textit{Proof of claim 2:} 
We use induction on the $rank(A(\Gamma))$, for $rank(A(\Gamma))=1$ we have that $A(\Gamma)$ is 
the integers $\mathbb{Z}$, so the statement holds. We assume that the statement holds for every $r \leq m$, we will show that holds for $m+1$ as well. Let $\Gamma$ be a connected simplicial graph with $m+1$ vertices.\\

Let $Sim(\Gamma)=n$, then $\Gamma$ contains the 1-skeleton of the standard $(n-1)$-simplex $S_{n-1}$ ($S_{n-1}  =  1-skel(\Delta^{n-1})).$\\
\textbf{Case 1.} $\Gamma = S_{n-1}.$\\
Then $m+1=n$, so by lemma \ref{2.11} we have 

$asdimA(S_{n-1}) \leq Val(S_{n-1}) +1$. By observing that $Val(S_{n-1}) = n-1$ we obtain that

\begin{center}
$asdimA(S_{n-1}) \leq n = Sim(\Gamma)$.
\end{center}
\textbf{Case 2.} $   S_{n-1} \subsetneqq  \Gamma.$\\
We will remove a vertex $u \in V(S_{n-1})$. Let's denote by $v_{i}$ ($i \in \lbrace 1,..., k \rbrace$) the vertices of $\Gamma$ which are adjacent to $u$ . We set $\Gamma^{\prime}=\Gamma \setminus \lbrace u \rbrace $. Obviously  $Sim(\Gamma^{\prime}) \leq n$. \\
We denote 
by $Y$ the full subgraph of $\Gamma$ formed from $\lbrace  v_{1}, \ldots ,v_{k} \rbrace$.\\

We observe that the RAAG $A(\Gamma)$ is an HNN-extension of the RAAG $A(\Gamma^{\prime})$. To be more precise, we have that
\begin{center}
$A(\Gamma)=  A(\Gamma^{\prime})\ast_{A(Y)}.$ 
\end{center}

By Theorem \ref{2.9} we obtain that 

\begin{equation}\label{eq2}
asdimA(\Gamma) \leq max\lbrace asdimA(\Gamma^{\prime})  , asdimA(Y)+1  \rbrace .
\end{equation}

Since $Sim(\Gamma^{\prime}) \leq n$ and $rank(\Gamma^{\prime}) \leq m$, by the inductive assumption we have that: 
\begin{equation}\label{eq3}
asdimA(\Gamma^{\prime}) \leq Sim(\Gamma^{\prime}) \leq n.
\end{equation}

We observe that $Sim(Y) \leq n-1$ and $rank(Y) \leq m$, then by the induction we obtain
\begin{equation}\label{eq4}
asdim A(Y) + 1 \leq Sim(Y) +1 \leq n.
\end{equation}  

by (\ref{eq2}), (\ref{eq3}) and (\ref{eq4}) we conclude that

\begin{center}
$asdimA(\Gamma) \leq n = Sim(\Gamma) .$  
\end{center}

\end{proof}

\section{Asymptotic dimension of one-relator groups.}

\begin{thm}\label{3.1}
Let G be a finitely generated one relator group. Then 
\begin{center}
$asdimG \leq 2$.
\end{center}
\end{thm}

\begin{proof}

Let $G= \langle S \mid r \rangle$ be a presentation of $G$ where $S$ is finite and $r$ is a cyclically reduced word in $S \bigcup S^{-1}$. To omit trivial cases, we assume that $S$
contains at least two elements and $\mid r \mid > 0$, (we denote by $\mid r \mid$ the 
length of the relator $r$ in the free group $F(S)$).\\

We may assume that every letter of $S$ appears in $r$. Otherwise our group $G$ is isomorphic to a free product $H \ast F$ of a finitely generated one relator group $H$ with relator $r$ and generating set $S_{H} \subseteq S$ consisting of all letters which appear in $r$ and a free group $F$ with generating set the remaining letters of $S$. We recall that the asymptotic dmension of any finitely generated non-abelian free group is equal to one. Then $asdim G = max \lbrace asdim H, asdim F \rbrace = max \lbrace asdim H, 1 \rbrace$ (see \cite{BD04}).\\

We denote by $\epsilon_{r}(s)$ the exponent sum of a letter $s \in S$ in a word $r$ and by $oc_{r}(s)$ the minimum number of the positions of appearance of the elements of the set $ \lbrace s^{k}$ for some $0 \neq k \in \mathbb{Z} \rbrace $ in a cyclically reduced word $r$. For example, if $r=abcab^{10}a^{-2}c^{-1}$, then $oc_{r}(a)=3$, $oc_{r}(b)=2$, $oc_{r}(c)=2$ and $\epsilon_r (c)=0$.

 We observe that if there exists $b \in S$ such that $oc_{r}(b)=1$, then the group $G$ is free (see \cite{LynSch}, thm 5.1, page 198 ), so $asdim G =1$. From now on we assume that for every $s \in S$ we have that $oc_{r}(s) \geq 2$ (so $\mid r \mid\geq 4$).\\

The proof is by induction on the length of $r$. We observe that if $\mid\!r\!\mid = 4$ then the statement of the theorem holds since by the result of D.Matsnev \cite{Mats} we have that $asdim G \leq  \frac{\lfloor \mid\!r\!\mid \rfloor}{2} = \frac{4}{2} =2$ (where $\lfloor \ast \rfloor$ is the floor function).\\
We assume that the statement of the theorem holds for all one relator groups with relator length smaller than or equal to $\mid\!r\!\mid - 1$.

We follow the arguments of McCool, Schupp and Magnus (see \cite{LynSch}, thm 5.1, page 198), which we shall describe in what follows. We distinguish two cases.\\
We note that the argument we use in case 1 is slightly different from the argument of McCool and Schupp as we will consider HNN-extensions over finitely generated groups. To be more precise, the classical arguments prove that if the relation of $G$ has exponent sum zero, then $G$ is an HNN-extension of another one relator group over an infinitely generated non-abelian free subgroup. Our contribution here is that we showed that $G$ is an HNN-extension of another one relator group over a \emph{finitely} generated non-abelian free subgroup.\\
In case 2 with non-zero exponent sum we use the original argument of Magnus to show that $G$ can be embedded in an one relator group $\Gamma$ whose defining relator has exponent sum zero.\\

\textbf{Case 1}: There exists a letter $a \in S$ such that $\epsilon_{r}(a)=0$.\\

We shall exhibit $G$ as an HNN-extension of a one relator group $G_{1}$ whose defining relator has shorter length than $r$, over a finitely generated free subgroup $F$.\\ 
Let  $S = \lbrace a=s_{1}, s_{2}, s_{3}, s_{4}, . . . s_{k} \rbrace $. Set $s^{(j)}_{i} = a^{j} s_{i} a^{-j}$ for $j \in \mathbb{Z}$ and  for $k \geq i \geq 2$. Rewrite $r$ scanning it from left to right and changing any occurrence of
$a^{j} s_{i}$ to $s^{(j)}_{i} a^{j}$, collecting the powers of adjacent $a-$letters together and continuing with the leftmost occurrence of $a$ or its inverse in the modified word.\\
We denote by $r^{\prime}$ the modified word in terms of $s^{(j)}_{i}$. 
We note that by doing this we make at least one cancellation of $a$ and its inverse.
The resulting word $r^{\prime}$ which represents $r$ in terms of $s^{(j)}_{i}$ and their inverses has length smaller than or equal to $\mid r \mid - 2$.\\
For example, if $r=as_{2}s_{3}as_{2}^{4}a^{-2}s_{3}$ then $r^{\prime}=s^{(1)}_{2}s^{(1)}_{3}(s^{(2)}_{2})^{4}s^{(0)}_{3}$.\\

Let $m$ and $M$ be the minimal and the maximal superscript of all $s_{i}^{(j)}$ ($i \geq 2$) occurring in  $r^{\prime}$ respectively.
To be more precise,
\begin{center}
$m=min \lbrace j \mid s_{i}^{(j)}$ occurs in $r^{\prime} \rbrace$ and $M=max \lbrace j \mid s_{i}^{(j)}$ occurs in $r^{\prime} \rbrace.$
\end{center}

Continuing our example, we have $m=0$ and $M=2$.\\

\textbf{Claim 1.1:} In case 1 we have $M-m>0$ and $m \leq 0 \leq M$.\\
We may assume, replacing $r$ with a suitable permutation if necessary, that $r$ begins with $a^{k}$ for some $k\neq 0$.
Then 
we can write $r=a^{k}swa^{n}tz$, where $k,n\neq 0$, $a \notin \lbrace s,t \rbrace \subseteq S$ and both $a$ and $a^{-1}$ do not appear in the word $z$ ($oc_z (a)=0$).\\
Then we observe that the letter $s$ has as superscript $k$ in the word $r^{\prime}$ while $t$ has as superscript $0$ in the word $r^{\prime}$. Since $k\neq 0$ we have that $M-m>0$. This completes the proof of the claim 1.1. \\

\textbf{Claim 1.2:} We claim that $G$ has a presentation

\begin{center}
$\langle a, s^{(j)}_{i},(i \in \lbrace 2 , \ldots , k \rbrace ) ,   (j \in \lbrace m,  \ldots , M \rbrace ) \mid r^{\prime} $, $ a s^{(j^{\prime})}_{i} a^{-1}(s^{(j^{\prime}+1)}_{i})^{-1} $ $  (j^{\prime} \in \lbrace m,  \ldots , M-1 \rbrace )   \rangle$.
\end{center}

To verify the claim, let $H$ be the group defined by the presentation given above. The map $\phi :G \longrightarrow H$ defined by 

\begin{center}
$a \longmapsto a$, $s_{i} \longmapsto s_{i}^{(0)}$
\end{center}

is a homomorphism since $\phi(r)=r^{\prime}$. On the other hand, the map $\psi :H \longrightarrow G$
defined by 

\begin{center}
$a \longmapsto a$, $s_{i}^{(j)} \longmapsto a^{j}s_{i}a^{-j}$
\end{center}

is also a homomorphism since all relators of $H$ are sent to $1_{G}$.\\ 

It is easy to verify that $\psi \circ \phi$ is the identity map of $G$.
The  homomorphism $\phi \circ \psi : H \rightarrow H$ maps $a \mapsto a$, $s_{i}^{(0)} \mapsto s_{i} \mapsto s_{i}^{(0)}$ and  $s_{i}^{(j)} \mapsto a^{j}s_{i}a^{-j} \mapsto a^{j}s_{i}^{(0)}a^{-j}$.\\
Now we show that $s_{i}^{(j)} = a^{j}s_{i}^{(0)}a^{-j}$. We have 

\begin{center}
$ a^{1}s_{i}^{(0)}a^{-1}=s_{i}^{(1)}$\\
$ a^{1}s_{i}^{(1)}a^{-1}=s_{i}^{(2)}$\\
$\ldots$\\
$ a^{1}s_{i}^{(j-1)}a^{-1}=s_{i}^{(j)},$
\end{center}

we combine these equations and we get $s_{i}^{(j)} = a^{1}s_{i}^{(j-1)}a^{-1}= a^{2}s_{i}^{(j-2)}a^{-2}= \ldots a^{j}s_{i}^{(0)}a^{-j}$, so $\phi \circ \psi =id_{H}$.\\

Since $\phi \circ \psi$ and $\psi \circ \phi$ are the identity maps on $H$ and $G$ respectively we deduce that $\phi$ is an isomorphism. This completes the proof of claim 1.2.\\



We set
\begin{center}
$G_{1}=\langle s^{(j)}_{i},(i \in \lbrace 2 , \ldots , k \rbrace ) ,   (j \in \lbrace m,  \ldots , M \rbrace )   \mid r^{\prime} \rangle$.
\end{center}

We note that there exists a letter $s_{i_{m}} \in S$ such that $s_{i_{m}}^{(m)}$ appears in $r^{\prime}$ and a letter $s_{i_{M}} \in S$ such that $s_{i_{M}}^{(M)}$ appears in $r^{\prime}$.\\

Now let $F$ and $\Lambda$ be the subgroups of $G_{1}$ generated respectively by the set $X = \lbrace s^{(j)}_{i},(i \in \lbrace 2 , \ldots , k \rbrace ) ,   (j \in \lbrace m,  \ldots , M-1 \rbrace )\rbrace$ and the set $Y = \lbrace s^{(j)}_{i},(i \in \lbrace 2 , \ldots , k \rbrace ) ,   (j \in \lbrace m+1,  \ldots , M \rbrace )\rbrace$.\\
\textbf{Claim 1.3:} The groups $F$ and $\Lambda$ are free subgroups of $G_1$.\\
This claim follows by the Freiheitssatz (see \cite{LynSch}, thm 5.1, page 198), since $X$ omits a generator of $G_{1}$ occurring in $r^{\prime}$ (this is the letter $s_{i_{M}}^{(M)}$) the subgroup $F$ is free.
The same holds for $\Lambda$, since $Y$ omits the letter $s_{i_{m}}^{(m)}$.\\
\textbf{Claim 1.4:} We have that $G \simeq G_{1}\ast_{F}$.\\
In particular, the map $s^{(j)}_{i} \longmapsto s^{(j+1)}_{i}$ from $X$ to $Y$ extends to an isomorphism from $F$ to $\Lambda$.\\
Thus $H$ is exhibited as the HNN extension of $G_{1}$ over the finitely generated free group $F$ using $a$ as a stable letter. Since $G \simeq H $ (claim 1.2) we have that 

\begin{center}
$G \simeq G_{1}\ast_{F}$.
\end{center}

By the fact that $\mid\!r^{\prime}\!\mid < \mid\!r\!\mid$ and the inductive assumption we have that $asdimG_{1} \leq 2$.\\
To conclude we apply the inequality for HNN-extensions (Theorem \ref{2.9}):\\
 $asdim G \leq max \lbrace asdim G_{1} , asdimF +1 \rbrace = max \lbrace asdim G_{1} , 2 \rbrace = 2$.\\

\textbf{Case 2}: For every letter $s \in S$ we have $\mid \epsilon_{r}(s) \mid \geq 1$.\\
Let $S = \lbrace a=s_{1}, b= s_{2}, s_{3}, s_{4}, . . . s_{k} \rbrace $ and $S_1 = \lbrace t,x,s_{i}, (3\leq i \leq k) \rbrace $. We consider the following homomorphism between the free group $F(S)$ and the free group $F(S_1)$
\begin{equation}\label{eq5}
\phi : a \longmapsto t^{-\epsilon_{r}(b)} x , b \longmapsto t^{\epsilon_{r}(a)} ,s_{i} \longmapsto s_{i} (3 \leq i \leq k).
\end{equation}
We set 
\begin{center}
$\Gamma=<S_1 \vert $ $ r(t,x,s_{i}, (3\leq i \leq k))>,$
\end{center}


where we denote by $r(t,x,s_{i},... (i>2))$ the modified word in terms of $t,x,s_{i}, (3\leq i \leq k)$ which is obtained from $r$ when we replace a generator $s$ with $\phi(s)$. Then $\phi$ induces a homomorphism $$\phi : G \rightarrow \Gamma.$$

The following claim shows that the homomorphism $\phi$ is actually a monomorphism into $\Gamma$, so we have an embedding of $G$ into $\Gamma$ via $\phi$:\\
\textbf{Claim 2.1:} The homomorphism $\phi : G \rightarrow \Gamma$ is monomorphism.\\ 
\textit{Proof of the claim:} We set $S_{2} = \lbrace  a,t,s_{i}, (3\leq i \leq k) \rbrace$ and $S_{1} = \lbrace x,t,s_{i}, (3\leq i \leq k)  \rbrace$. We define $g : F(S) \rightarrow F(S_{2})$ and $f : F(S_{2}) \rightarrow F(S_{1})$, by

\begin{center}
$g : a \longmapsto a , b \longmapsto t^{\epsilon_{r}(a)} ,s_{i} \longmapsto s_{i} (3 \leq i \leq k)$,
\end{center}

\begin{center}
$f : a \longmapsto t^{-\epsilon_{r}(b)} x , t \longmapsto t ,s_{i} \longmapsto s_{i} (3 \leq i \leq k)$.
\end{center}

We set: $r_{2}=g(r)$ , $G_{2}= <S_{2} \mid r_{2}>$, $r_{1}=f \circ g (r)=r(t,x,s_{i}, (3\leq i \leq k))$ and we observe that $\Gamma= <S_{1} \mid r_{1}>$. Then $g$ induces a homomorphism $\overline{g}: G \mapsto G_{2}$ and $f$ induces a homomorphism $\overline{f}: G_{2} \mapsto \Gamma$. Obviously, $\phi = \overline{f} \circ \overline{g}$.\\

We can easily see that $\overline{f}$ is an isomorphism. Indeed, the homomorphism $\psi : \Gamma \rightarrow G_2$ given by

\begin{center}
$ x \longmapsto t^{\epsilon_{r}(b)} a , t \longmapsto t ,s_{i} \longmapsto s_{i} (3 \leq i \leq k) $
\end{center}
is the inverse homomorphism of $ \overline{f}$.

It is enough to prove that $\overline{g}$ is monomorphism. This follows by the fact that the group $G_{2}$ is the amalgamated product $G \ast_{\mathbb{Z}} <t>$, where $\mathbb{Z} = < \lambda >$ and $\psi_{1}(\lambda)= b$ , $\psi_{2}(\lambda)= t^{\epsilon_{r}(a)}$ are the corresponding monomorphisms. We can see that the homomorphism $\overline{g}$ is the inclusion of $G$ into the amalgamated product, so $\overline{g}$ is injective. This completes the proof of the claim 2.1.\\


 We denote by $r(t,x,s_{i},... (i \geq 3))$ the modified word in terms of $t,x,s_{i}, (3\leq i \leq k)$ which can be obtained from $r$ when we replace a generator $s$ with $\phi(s)$ and by $p$ the cyclically reduced  $r(t,x,s_{i}, (3\leq i \leq k))$.\\ 
We observe that $\epsilon_{p}(t)=0$ and that $x$ occurs in $p$.\\

 If the letter $t$ occurs in the word $p$, from Case 1 we have that $\Gamma$ is an HNN extension of some group $H$ over a free subgroup $F$, namely, $\Gamma=H\ast_{F}$. \\
As in Case 1 by assuming
that $p$ starts with $t$ or $t^{-1}$ we introduce new variables $s^{(j)}_{i} = t^{j}s_{i}t^{-j}$. Using these variables, we rewrite $p$ as a word $w$, eliminating all occurrences of $t$ and its inverse.
Then we observe that $ \mid\!w\!\mid \leq \mid\!r\!\mid - 1$. By using the inductive assumption for
$w$ we obtain that $$asdimG \leq asdim\Gamma \leq 2.$$

If the letter $t$ does not occur in the word $p$, we observe that $$\mid\!p\!\mid \leq \mid\!r\!\mid - 1.$$ Then $$\Gamma = \langle t \rangle \ast \Gamma^{\prime},$$ where $$\Gamma^{\prime}= \langle x,s_{i}, (3\leq i \leq k) \mid p \rangle.$$ 

Since $asdim ( G_{1} \ast G_{2} ) = max \lbrace asdim G_{1}, asdim G_{2} \rbrace$ holds (see \cite{BD04}) we have that

$$asdim\Gamma = max \lbrace 1 , asdim \Gamma^{\prime} \rbrace.$$

Then by the
inductive assumption for
$p$ we have that $asdim\Gamma^{\prime} \leq 2$. Finally, we conclude that $$asdimG \leq  asdim\Gamma \leq 2.$$


\end{proof}

\subsection{One relator groups with asymptotic dimension two.}
We recall that a nontrivial group $H$ is \textit{freely indecomposable} if $H$ can not be expressed as a free product of two non-trivial groups.

A natural question derived from Theorem \ref{3.1} is, which one relator groups have asymptotic dimension two.\\
In this subsection, we will show that the asymptotic dimension of every finitely generated one relator group that is not a free group or a free product of a free group and a finite cyclic group is exactly two.

We will use following propositions \ref{3.2} and \ref{3.3} from \cite{FKS} and \cite{S} respectively.
\begin{prop}\label{3.2}
Let $G$ be an infinite finitely generated one relator group with torsion. If $G$ has more than one ends, then $G$ is a free product of a nontrivial free group and a freely indecomposable one relator group. 
\end{prop}

\begin{prop}\label{3.3}
Let $G$ be a torsion free infinite finitely generated group. If $G$ is virtually free, then it is free.
\end{prop}

\begin{lem}\label{3.4}
Let $G$ be an infinite finitely generated one relator group such that is not a free group or a free product of a nontrivial free group and a freely indecomposable one relator group. Then $G$ is not virtually free.
\end{lem}
\begin{proof}
If $G$ has torsion, by proposition \ref{3.2} we have that $G$ has exactly one end, so $G$ can not be virtually free. If $G$ is torsion free, by proposition \ref{3.3} we obtain that $G$ is free and this is a contradiction by the assumption of the lemma.
\end{proof}

We note that every finite one relator group is cyclic. To see that it is enough to observe that every one relator group with at least two generators has infinite abelianization.

The following proposition is the main result of this subsection.

\begin{prop}\label{3.5}
Let $G$ be a finitely generated one relator group that is not a free group or a free product of a free group and a finite cyclic group. Then 
\begin{center}
$asdim\,G = 2$.
\end{center}
\end{prop}

\begin{proof}
By Theorem \ref{3.1} we have that $asdim\,G \leq 2$. If $G$ is finite then it is cyclic.
If $G$ is infinite we have that $1 \leq asdim\,G$. By a theorem of Gentimis (\cite{Ge}) we have that  $asdim\,G=1$ if and only if $G$ is virtually free.\\
We assume that $G$ is an infinite virtually free group. So if $G$ is torsion free then by proposition \ref{3.3} we obtain that $G$ is free. If $G$ has torsion then by lemma \ref{3.4} $G$ is a free product
of a nontrivial free group and a freely indecomposable one relator group $G_1$. Observe that $G_1$ is an infinite noncyclic group, then by the same lemma $G_1$ is not 
virtually free so $G$ is not virtually free either, which is a contradiction.\\
We conclude that $asdim\,G = 2$.
\end{proof}
\begin{cor} \textit{Let $G$ be a finitely generated freely indecomposable one relator group which is not cyclic. Then
\begin{center}
$asdim\,G = 2$. 
\end{center}}
\end{cor}

The following proposition can be found in \cite{LynSch} (prop. 5.13, page 107).
\begin{prop}\label{3.6}
Let $G=\langle x_1, \ldots ,x_n \vert r \rangle$ be a finitely generated one relator group, where $r$ is of minimal lenght under $Aut(F(\lbrace x_1, \ldots ,x_n \rbrace))$ and contains exactly the generators $x_1, \ldots ,x_k$ for some $k$, $0 \leq k \leq n$. Then $G$ is isomorphic to the free product $G_1 \ast G_2$, where $G_1=\langle x_1, \ldots ,x_k \vert r \rangle$ is freely indecomposable and $G_2$ is free with basis $\lbrace x_{k+1}, \ldots ,x_{n} \rbrace$.

\end{prop}

Using the above results we sum up to following corollary which describes the finitely generated one relator groups.
\begin{cor} \textit{Let $G$ be a finitely generated one relator group. Then one of the following is true}:\\
\textbf{(i)} \textit{$G$ is finite cyclic, and $asdim\,G = 0$} \\
\textbf{(ii)} \textit{$G$ is a nontrivial free group or a free product of a nontrivial free group and a finite cyclic group, and $asdim\,G = 1$}\\
\textbf{(iii)} \textit{$G$ is an infinite freely indecomposable not cyclic group or a free product of a nontrivial free group and an infinite freely indecomposable not cyclic group, and $asdim\,G = 2$.}
\end{cor}
We could further describe the boundaries of hyperbolic one relator groups.
We recall the following result of Buyalo and Lebedeva (see \cite{BL}) for hyperbolic groups:
\begin{center}
$asdim G = dim \partial_{\infty}G + 1$.
\end{center}   
  
Let $G$ be an infinite finitely generated hyperbolic one relator group that is not virtually free. By T.Gentimis (\cite{Ge}) we obtain that $asdim G \neq 1$,so $asdim G =2$. Using the previous equality we obtain that $G$ has one dimensional boundary.\\ Applying a theorem of M. Kapovich and B. Kleiner (see \cite{KK}) we can describe the boundaries of hyperbolic one relator groups.

\begin{prop}\label{3.7}
Let $G$ be a hyperbolic one relator group. Then $asdim\,G= 0 ,1 $or $2$.\\ 
\textbf{(i)} If $asdim\,G=0$, then $G$ is finite.\\
\textbf{(ii)} If $asdim\,G=1$, then $G$ is virtually free and the boundary is a Cantor set.\\
\textbf{(iii)} If $asdim\,G=2$ providing that $G$ does not split over a virtually
cyclic subgroup, then one of the following holds:\\
1. $\partial_{\infty}G$ is a Menger curve.\\
2. $\partial_{\infty}G$ is a Sierpinski carpet.\\
3. $\partial_{\infty}G$ is homeomorphic to $S^{1}$.
\end{prop}

\section{Graphs of Groups.}
We will prove a general theorem for the asymptotic dimension of fundamental groups of finite graphs of groups.

\begin{thm}\label{4.1}
Let $(\mathbb{G}, Y)$ be a finite graph of groups with vertex groups $\lbrace G_{v} \mid v \in Y^{0} \rbrace$ and edge groups $\lbrace G_{e} \mid e \in Y^{1}_{+} \rbrace$. Then the following inequality holds:
\begin{center}
$asdim \pi_{1}(\mathbb{G},Y,\mathbb{T})  \leq max_{v \in Y^{0} ,e \in Y^{1}_{+}} \lbrace asdim G_{v}, asdim\,G_{e} +1 \rbrace.$
\end{center}

\end{thm}

\begin{proof} We use induction on the number $\sharp E(Y)$ of edges of the graph $Y$. For $\sharp E(Y)=1$ we distinguish two cases. The first case is when the fundamental group  $\pi_{1}(\mathbb{G},Y,\mathbb{T})$ is an amalgamated product, so the statement of the theorem follows by the following inequality of A.Dranishnikov  (see \cite{Dra08}) 

  \begin{center}
$asdimA \ast_{C} B \leq max \lbrace asdim A, asdim B, asdimC +1 \rbrace.$
\end{center}

The second case is when the fundamental group  $\pi_{1}(\mathbb{G},Y,\mathbb{T})$ is an HNN-extension, so the statement of the theorem follows by Theorem \ref{2.9}.\\

We assume that the statement of the theorem holds for $E(Y) \leq m$. Let $(\mathbb{G}, Y)$ be a finite graph of groups with $\sharp E(Y) = m+1$. We denote by $\mathbb{T}$ a maximal tree of $Y$.


We distinguish two cases:\\
\textbf{Case 1:} $Y= \mathbb{T}$. We remove a terminal edge $e^{\prime}=[v,u]$ from the graph $Y$ such that the full subgraph of $Y$ denoted by $\Gamma$ and formed from the vertices  $V(Y) \setminus \lbrace  u \rbrace$ is connected.
We observe that $\Gamma$ is also a tree which we denote by $\mathbb{T}^{\prime}$.

Then $\pi_{1}(\mathbb{G},Y,\mathbb{T})= \pi_{1}(\mathbb{G},\Gamma,\mathbb{T}^{\prime}) \ast_{G_{e^{\prime}}} G_{u}$, so by the inequality for amalgamated products of A.Dranishnikov (see \cite{Dra08}), we have
\begin{center}
$asdim \pi_{1}(\mathbb{G},Y,\mathbb{T}) \leq max \lbrace  asdim\pi_{1}(\mathbb{G},\Gamma,\mathbb{T}^{\prime}) ,  asdim\,G_{u} , asdim G_{e^{\prime}}+1  \rbrace$.
\end{center}
Since $\sharp E(\Gamma) = m$, by the inductive assumption we obtain that 
\begin{center}
$asdim \pi_{1}(\mathbb{G},\Gamma,\mathbb{T}^{\prime})  \leq max_{v \in Y^{0}\setminus \lbrace u \rbrace ,e \in Y^{1}_{+}\setminus \lbrace e^{\prime} \rbrace} \lbrace asdim G_{v}, asdim\,G_{e} +1 \rbrace$,
\end{center} so 
\begin{center} 
$asdim \pi_{1}(\mathbb{G},Y,\mathbb{T})  \leq max_{v \in Y^{0} ,e \in Y^{1}_{+}} \lbrace asdim G_{v}, asdim\,G_{e} +1 \rbrace.$
\end{center}
\textbf{Case 2:} $\mathbb{T} \subsetneqq Y $. We remove from $Y$ an edge $e^{\prime}=[v,u]$ which doesn't belong to $\mathbb{T}$.\\
Since the tree $\mathbb{T}$ is a maximal tree of $Y$ and $e^{\prime} \not\in E(\mathbb{T})$ we have that the graph $\Gamma=Y  \setminus e^{\prime}$ is connected and $\mathbb{T} \subseteq \Gamma$.\\
Then $\pi_{1}(\mathbb{G},Y,\mathbb{T})= \pi_{1}(\mathbb{G},\Gamma,\mathbb{T}) \ast_{G_{e^{\prime}}}$, so by the inequality for HNN-extensions (Theorem \ref{2.9}) we have
\begin{center}
$asdim \pi_{1}(\mathbb{G},Y,\mathbb{T}) \leq max \lbrace  asdim\pi_{1}(\mathbb{G},\Gamma,\mathbb{T}), asdim G_{e^{\prime}}+1  \rbrace$.
\end{center}
Since $\sharp E(\Gamma) = m$, by the inductive assumption we obtain that 
\begin{center}
$asdim \pi_{1}(\mathbb{G},\Gamma,\mathbb{T})  \leq max_{v \in Y^{0},e \in Y^{1}_{+}\setminus \lbrace e^{\prime} \rbrace} \lbrace asdim G_{v}, asdim\,G_{e} +1 \rbrace$,
\end{center} so 
\begin{center}
$asdim \pi_{1}(\mathbb{G},Y,\mathbb{T})  \leq max_{v \in Y^{0} ,e \in Y^{1}_{+}} \lbrace asdim G_{v}, asdim\,G_{e} +1 \rbrace.$
\end{center}





\end{proof}
 
We obtain as a corollary the following proposition.
\begin{prop}
Let $(\mathbb{G}, Y)$ be a finite graph of groups with \begin{center}
vertex groups $\lbrace G_{v} \mid v \in Y^{0} \rbrace$ and edge groups $\lbrace G_{e} \mid e \in Y^{1}_{+} \rbrace$.
\end{center}
 We assume that $max_{e \in Y^{1}_{+}} \lbrace  asdim\,G_{e}  \rbrace < max_{v \in Y^{0} } \lbrace asdim G_{v} \rbrace = n $. Then,

\begin{center}
$asdim \pi_{1}(\mathbb{G},Y,\mathbb{T})  = n.$
\end{center}

\end{prop}

\textit{E-mail}: panagiotis.tselekidis@queens.ox.ac.uk


\textit{Address:} Mathematical Institute, University of Oxford, Andrew Wiles Building, Woodstock Rd, Oxford OX2 6GG, U.K.


\end{document}